\documentclass[11pt,reqno,a4paper]{amsart}
\usepackage[a4paper, total={6in, 9in}]{geometry}
\usepackage[utf8]{inputenc}

\usepackage[usenames]{color}
\usepackage[abbrev,nobysame,alphabetic]{amsrefs}
\usepackage{ amssymb }
\usepackage{amsmath,pdfsync,verbatim,graphicx,epstopdf,enumerate}
\usepackage{mathtools} 
\mathtoolsset{showonlyrefs} 
\usepackage{thmtools}
\usepackage{thm-restate}
\usepackage[colorlinks=true]{hyperref}
\usepackage{cancel}
\usepackage{accents}
\usepackage[framemethod=tikz]{mdframed}
\usepackage{comment}
\usepackage{setspace}

\allowdisplaybreaks
\hypersetup{linkcolor=red,citecolor=green}
\numberwithin{equation}{section}
\numberwithin{equation}{section}

\newtheorem{theorem}{Theorem}[section]

\newtheorem{corollary}[theorem]{Corollary}
\newtheorem{lemma}[theorem]{Lemma}
\newtheorem{proposition}[theorem]{Proposition}
\theoremstyle{definition}

\newtheorem{remark}[theorem]{Remark}

\newcommand{\Om}{\Omega}

\newcommand{\D}{\Delta}
\newcommand{\e}{\varepsilon}

\newcommand{\p}{\partial}
\renewcommand{\D}{\mathrm{d}}

\newcommand{\Rb}{\mathbb{R}}

\newcommand{\Rnn}{\mathbb{R}^n}

\newcommand{\ol}[1]{\overline{#1}}
\newcommand{\norm}[1]{\left\lVert #1 \right\rVert}
\newcommand{\abs}[1]{\lvert #1 \rvert}

\newcommand{\dst}{\D s\,\D t}
\newcommand{\dstu}{\D\tau\, \D s\, \D t}

\newcommand{\baligned}{\begin{equation}\begin{aligned}}
	\newcommand{\ealigned}{\end{aligned}\end{equation}}
	
	\newcommand{\V}{\mathcal{V}}

	\newcommand{\janne}[1]{\begin{quotation}\textbf{\color{purple}Janne's comment:\ }{\color{purple}\textit{#1}}\end{quotation}}

	\title{An inverse problem for a nonlinear biharmonic operator}
	\author[Nurminen]{Janne Nurminen$^{\dagger}$}
	\address{$^{\dagger}$Computational Engineering, School of Engineering Sciences, Lappeenranta-Lahti University of Technology, Finland \& Department of Mathematics and Statistics, University of Jyv\"askyl\"a, Jyv\"askyl\"a, Finland}
	\email{janne.nurminen@lut.fi, janne.s.nurminen@jyu.fi}

	\author[Sahoo]{Suman Kumar Sahoo$^{\mathsection}$}
	\address{$^{\mathsection}$Department of Mathematics, Indian Institute of Technology, Bombay, India}
	\email{suman@math.iitb.ac.in,sumansahootifr@gmail.com}
	
\begin{document}
\begin{abstract}
	An inverse problem for a nonlinear biharmonic operator is under consideration in the spirit of \cite{isakov_1993,JNS2023}. We prove that a general nonlinear term of the $Q= Q(x,u, \nabla u, \Delta u)$ associated to a nonlinear biharmonic operator can be recovered from the local Cauchy data set. The proof uses first order linearization method, Runge approximation, and uniqueness results for the linearized inverse problem. 
\end{abstract}


\maketitle
\section{introduction and main results}
Let \(\Omega\) be a bounded domain with smooth boundary in \(\mathbb{R}^n\), where \(n \geq 2\), and let
\begin{align}\label{eq_reg_Q}
	Q = Q(x, u, \nabla u, \Delta u) = Q(x, z, X, q) \in C^k\left(\mathbb{R}, C^{4,\alpha}(\overline{\Omega}), C^{4,\alpha}(\overline{\Omega}; \mathbb{R}^n), C^{4,\alpha}(\overline{\Omega})\right), \quad k \geq 3. 
\end{align}

With this notation, consider the following semilinear fourth-order biharmonic operator:
\begin{equation}\label{main_equation}
	\begin{aligned}
		\mathcal{L}_{Q} u := \Delta^2 u + Q(x, u, \nabla u, \Delta u) &= 0 \quad \text{in } \Omega,
	\end{aligned}
\end{equation}
associated with the Navier boundary condition \((u, \Delta u) = (f_0, f_1)\) on \(\partial\Omega\). The boundary measurements corresponding to \eqref{main_equation} can be encoded in terms of the Navier-Neumann map as follows:
\[
\mathcal{N}_Q(f_0, f_1) := (\partial_{\nu} f_0, \partial_{\nu} f_1) \quad \text{for small boundary data } f_0, f_1.
\]
Alternatively, one could consider the boundary measurements in terms of Cauchy data sets:
\[
C_{Q,\mathcal{N}} = \left\{(u|_{\partial\Omega}, \partial_{\nu}u|_{\partial\Omega}, \Delta u|_{\partial\Omega}, \partial_{\nu} \Delta u|_{\partial\Omega}) : u \text{ solves } \mathcal{L}_{Q} u = 0 \text{ in } \Omega \right\}.
\]
The inverse problem we are interested in is to recover the nonlinear term \(Q\) from the knowledge of the Cauchy data set \(C_{Q}\). More precisely, we work with the local Cauchy data. Let \(\delta > 0\) and \(w \in C^{4,\alpha}(\overline{\Omega})\) be a fixed solution of \eqref{main_equation}. Then the local Cauchy data set is defined as:
\[
C_{Q}^{w, \delta} = \left\{(u|_{\partial\Omega}, \partial_{\nu}u|_{\partial\Omega}, \Delta u|_{\partial\Omega}, \partial_{\nu} \Delta u|_{\partial\Omega}) : u \text{ solves } \mathcal{L}_Q u = 0 \text{ in } \Omega, \text{ and } \|u - w\|_{C^{4,\alpha}(\overline{\Omega})} \leq \delta \right\}.
\]
Our main result states that if \(w\) is a common solution of \(\mathcal{L}_{Q_j} u = 0\) for \(j = 1, 2\) and the corresponding Cauchy data sets satisfy \(C_{Q_1}^{w, \delta} \subset C_{Q_2}^{w, \delta}\), then \(Q_1 = Q_2\) near \(w\). More precisely, we have:

\begin{theorem}\label{Th:main}
	Let \(n \geq 2\) and \(w\) solve \(\mathcal{L}_{Q_1} w = \mathcal{L}_{Q_2} w = 0\) in \(\Omega\), where \(Q_1\) and \(Q_2\) satisfy \eqref{eq_reg_Q}. Suppose for some \(\delta, C > 0\) we have $	C_{Q_1}^{w, \delta} \subseteq C_{Q_2}^{0, C}.$
	Then there exists \(\epsilon > 0\) such that
	\[
	Q_1(x, w(x) + \lambda, \nabla w(x) + \overline{\lambda}, \Delta w(x) + \lambda) = Q_2(x, w(x) + \lambda, \nabla w(x) + \overline{\lambda}, \Delta w(x) + \lambda),
	\]
	for \(x \in \overline{\Omega}\), \(|\lambda| \leq \epsilon\), and \(\overline{\lambda} = (\lambda, \ldots, \lambda) \in \mathbb{R}^n\).
\end{theorem}

As a corollary, we have the following result:

\begin{corollary}
	Let \(n \geq 2\) and \(\Omega \subset \mathbb{R}^n\) be a bounded domain with smooth boundary. Suppose \(\mathcal{N}_{Q_1}(f_0, f_1) = \mathcal{N}_{Q_2}(f_0, f_1)\) for small \(f_0\) and \(f_1\). Then
	\[
	Q_1(x, u, \nabla u, \Delta u) = Q_2(x, u, \nabla u, \Delta u).
	\]
\end{corollary}

If two nonlinearities \(Q_1\) and \(Q_2\) do not have a common solution, then there is a gauge invariance for the inverse problem; see \cite{JNS2023}. If \(\phi \in C^{4,\alpha}(\overline{\Omega})\) satisfies \(\phi = 0\) on \(\partial\Omega\) and \(\partial^k_{\nu} \phi = 0\) on \(\partial\Omega\) for \(k = 1, 2, 3\), then
\[
\mathcal{L}_Q u = 0 \quad \text{in } \Omega \quad \text{if and only if} \quad \mathcal{L}_{T_{\phi}} v = 0 \quad \text{in } \Omega,
\]
where \(v = u - \phi\), and $T_{\phi} Q(x, u, \nabla u, \Delta u) = \Delta^2 \phi - Q(x, u + \phi, \nabla (u + \phi), \Delta (u + \phi)).$

Our next result shows that if one knows the Cauchy data for a nonlinearity \(Q\) and for solutions close to a given solution \(w\), then one can recover \(Q\) near points \((x, w(x))\) up to the gauge mentioned above.
\begin{theorem}\label{th_main_upto_gauge}
	Let \(w\) solve the equation \(\mathcal{L}_{Q_1} w = 0\) in \(\Omega\). Suppose for some \(\delta, C > 0\) we have
	$	C_{Q_1}^{w, \delta} \subseteq C_{Q_2}^{0, C}$, where \(Q_i\) satisfy \eqref{eq_reg_Q} for \(i = 1, 2\). 
	Then there exists \(\epsilon > 0\) such that
	\[
	Q_1(x, w(x) + \lambda, \nabla w(x) + \overline{\lambda}, \Delta w(x) + \lambda) = T_{\phi}Q_2(x, w(x) + \lambda, \nabla w(x) + \overline{\lambda}, \Delta w(x) + \lambda),
	\]
	for \(x \in \overline{\Omega}\), \(|\lambda| \leq \epsilon\), and \(\overline{\lambda} = (\lambda, \ldots, \lambda) \in \mathbb{R}^n\).
\end{theorem}

\subsection{State of the Art}

The linearization of nonlinear partial differential equations (PDEs) is a well-established and powerful method for simplifying and solving complex problems. This method involves linearizing a nonlinear PDE near a given solution, using the inverse/implicit function theorem, making it easier to analyze and solve locally. We employ similar ideas to solve our inverse problem. 

The inverse problem for a linear second-order elliptic equation was initially posed by Calder\'on for the conductivity equation in 1980 \cite{Calderon_1980}, where he established results for the linearized equation. Since then, numerous researchers have explored the linearized Calder\'on problem in various settings; see \cite{Calderon_1980}. In 1987, Sylvester and Uhlmann \cite{sylvester1987} solved the nonlinear inverse problem for a linear second-order elliptic PDE. There is a substantial body of literature available on inverse problems related to second-order linear elliptic PDEs, and we refer to the survey \cite{uhlmann2009} for further details.

Inverse problems involving nonlinear PDEs began with Isakov's work \cite{isakov_1993}. In \cite{isakov_1993}, he demonstrated that a nonlinear term $
a(x,z)$  in a nonlinear Schr\"odinger equation $\Delta u +a(x,u)=0$ can be recovered from the Dirichlet-to-Neumann map $\Lambda_a=\p_{\nu} u|_{\p\Om}$, using the first-order linearization technique under suitable assumptions on 
$a(x,z)$. Similar results can be found in \cite{isakov_sylvester_1994, isakov_1993}.
These works all relied on the following assumption:   either $a(x,0)=0$ or $ \p_u a(x,u)\le 0$. Such assumptions ensured well-posedness and the maximum principle, which were crucial in closing the argument. Subsequent studies, particularly \cite{isakov_nachman_1995,sun2010}, weakened these requirements. Notably, their results establish that the nonlinear term 
can be recovered in the reachable set
\begin{align}
	E_a=\{ (x,z): x\in \ol{\Om} \,\, \mbox{and} \,\,z=u(x)\,\, \mbox{for some solution of $ \Delta u +a(x,u)=0$} \},  
\end{align}
and  in general, the set $E_a$ cannot be all of $\ol{\Om}\times \Rb$; see the counterexample in \cite{isakov_1993}. Recently, \cite{JNS2023} eliminated the sign condition on 
$ a(x,u)$ and solved an inverse problem for the equation  $\Delta u +a(x, u)=0$, for a general nonlinear term
$a(x,u)$. For more results related to inverse problems for nonlinear equations, we refer to the survey \cite{sun2005}.

The higher-order linearization method was initially introduced by Kurylev, Lassas, and Uhlmann \cite{KLu-linear_hyperbolic} to tackle inverse problems for nonlinear wave equations. While the first-order linearization relies on techniques from linear theory, the higher-order approach provides a powerful framework for addressing nonlinear inverse problems that remain unresolved in the linear case
see \cite{Kian_Kru_Uhlmann, Kruchyk_Uhl, Lai_Zhou}, and the references therein.  Subsequent works \cite{LLLS_elliptic_nonlinear, FO_semilinear_elliptic, LLST_fractional_power} extended this method to nonlinear elliptic equations. Additionally, we highlight the works \cite{Catalin_Ghosh_Nakamura_aniso_poros_media, Catalin_Ghosh_Uhlmann_poros_media}, which explore inverse problems for porous media equations where nonlinearity appears in the principal part of the equation rather than lower-order terms. For further exploration of similar approaches, we refer to the  works \cite{Kian_JFA,Nuerminen_ms_2023,Nurminen_ms_2024,CLL_minimal_jde,KLL_Forum_sigma} and the references cited in these articles.

The study of inverse problems for higher-order elliptic equations is relatively limited compared to second-order elliptic equations. This area of research began with the works of Krupchyk, Lassas, and Uhlmann \cite{KLU_biharmonic,KLU_poly}. For more recent contributions, see \cite{BKS_poly,SS_linearizer_polyharmonic}. In a recent study \cite{BKSU_nonlinear_CPDE}, the authors investigated an inverse problem for a nonlinear biharmonic operator involving a third-order anisotropic perturbation. This is another instance where nonlinearity is used as a tool to solve inverse problems that remain unsolved for their linear counterparts. We refer to the works  \cite{GK_poly,BG_2019,BG_2022,Pranav_Kumar_local_data,AJS_biharmonic_partial} for more results, and discussion on inverse problems related to higher-order operators. To the best of our knowledge, there are no existing studies addressing inverse problems for nonlinear biharmonic operators where a generic nonlinear coefficient depends on 
$u,\nabla u,\Delta u$. We address this problem in this article using first order linearization.

Let us briefly describe the method used (for a more complete description we refer to \cite{JNS2023}). The key idea is to linearize the problem and construct solutions near a fixed solutions to the nonlinear problem. This is done in two parts. First we construct solutions that depend continuously on the solutions to the linearized equation. These will have a specific form that in applying estimates using Cauchy data. When we have two nonlinear equations $\Delta^2 u + Q_i(x,u,\nabla u, \Delta u)=0$, $i=1,2$, this first solution operator gives that the solutions will depend continuously on the corresponding linearized equations. But we would need the solutions to depend on the same linearized solutions, that is, for solutions on the linearized equation for $\Delta^2 u + Q_1(x,u,\nabla u, \Delta u)=0$. This is made possible by the second solution operator. Then we combine these solution operators with the mapping properties of the first solution operator and results for the linearized inverse problem to conclude the proof of our main results.

The rest of the article is organized as follows.
In Section \ref{sec:smooth_solution_map} we construct the first solutions operator. Section \ref{sec:estimating_Cauchy_data} is dedicated to Cauchy data estimates and in Section \ref{sec:controlling_Cauchy_data} we use these results to construct the second solution operator. In Section \ref{sec:first_lin} we show that the coefficients of the linearized equation agree for $i=1,2$. Section \ref{sec:proof_of_main} gives the proof of our main results and Section \ref{sec:runge} gives a Runge approximation result used in the proof of the main results. Finally, in the Appendix \ref{sec:appendix} we show some technical calculations used in Section \ref{sec:smooth_solution_map}.

\section*{acknowledgment}
JN and SKS are thankful to Mikko Salo for several fruitful discussions on the  topics related to the article. JN was supported by the Finnish Centre of Excellence in Inverse Modelling and Imaging (Academy of Finland Grant 284715), the Research Council of Finland (Flagship of Advanced Mathematics for Sensing Imaging and Modelling grant 359208) and by the Emil Aaltonen Foundation.


\section{A smooth solution map}\label{sec:smooth_solution_map}
The goal of this section is to provide a good solution map for implementing the linearization method. We address  the situation when the linearized problem is well-posed in some suitable function space. The  latter situation when the linearized problem is not well-posed can be dealt with using Fredholm alternatives, using the arguments used in \cite{JNS2023}.  Because for most of the Dirichlet data (up-to a finite dimension space) it is well-posed by the Fredholm theory.

We use the notation $ A \lesssim B \implies A \le C B$ for some constant $C$, throughout the rest of the article, and the value of that constant might change from line to line. We start with recalling the following regularity result from \cite[Theorem 2.19]{Gazzola_book}.
\begin{lemma}\label{lem_regularity}
	Let $F\in C^{\alpha} (\ol{\Om})$ and $(f_0,f_1)\in C^{4,\alpha}(\p\Om) \times C^{2,\alpha}(\p\Om)$. Then the following boundary value problem 
	\begin{equation}\label{eq:boundary_value_problem}
		\begin{aligned}
			( \Delta^2 + A \Delta + X\cdot \nabla + V)u &=F\quad \qquad \,\,\mbox{in $\Om$},\\
			(u, \Delta u)&=(f_0,f_1) \quad \mbox{on $\p\Om$},
		\end{aligned}
	\end{equation}
	admits a unique solution $u_{F,f_0,f_1}=G_{A,X,V}(F,f_0,f_1)$ and satisfies
	\begin{align}
		\norm{u_{F,f_0,f_1}}_{C^{4,\alpha} (\ol{\Om})} \lesssim \left( \norm{F}_{C^{\alpha} (\ol{\Om})} + \norm{f_0}_{C^{4,\alpha}(\p\Om)}+ \norm{f_1}_{C^{2,\alpha}(\p\Om)} \right),
	\end{align}
	where the implicit constant is independent of $F,f_0$ and $f_1$.
	
\end{lemma}
In the next Lemma  we study  the linearized equation using a fixed point argument. 
\begin{lemma}\label{lemma_fixed_point}
	Let $Q \in C^k(\Rb,C^{4,\alpha}(\ol{\Omega}), C^{4,\alpha}(\ol{\Omega};\Rnn),C^{4,\alpha}(\ol{\Omega}))$ and $w\in C^{k,\alpha}(\ol{\Om})$ be a solution of $\Delta^2 w + Q(x,w,\nabla w, \Delta w) = 0$ in $\Om.$
	Let $G_Q(F,f_0,f_1)$ be the solution operator of
	\begin{equation*}
		\begin{cases}
			\Delta^2 u + \partial_{q} Q(x,w,\nabla w,\Delta w) \Delta u + \nabla_p Q(x,w,\nabla w,\Delta w)\cdot\nabla u + \p_uQ(x,w,\nabla w,\Delta w)u = F& \quad\text{in }\Om \\
			(u, \Delta u)= (f_0,f_1)&\quad\text{on }\p\Om.
		\end{cases}
	\end{equation*}
	Define $R_v(r) = R(v+r)$, where $R:C^{4,\alpha}(\ol{\Om})\to C^{\alpha}(\ol{\Om})$ is given by
	\begin{equation}\label{def_capital_R}
		\begin{aligned}
			R(h) &:= \int_0^1 [\p_u Q(x,w+th,\nabla(w+th),\Delta(w+th)) - \p_u Q(x,w,\nabla w,\Delta w)]h \,\D t \\
			& \quad+ \int_0^1 [\nabla_p Q(x,w+th,\nabla(w+th),\Delta(w+th)) - \nabla_p Q(x,w,\nabla w,\Delta w)]\cdot \nabla h \,\D t\\
			&\quad+ \int_0^1 [\p_q Q(x,w+th,\nabla(w+th),\Delta(w+th)) - \p_q Q(x,w,\nabla w,\Delta w)] \Delta h \,\D t.
		\end{aligned}
	\end{equation}
	Let  $v\in C^{4,\alpha}(\ol{\Om})$ be fixed. Next we  define $T_v: C^{4,\alpha}(\ol{\Om})\to C^{4,\alpha}(\ol{\Om})$ by $T_v(r) = -G_Q(R_v(r), 0,0)$. Under the above assumptions, there exists $\delta>0$ such that the map $T_v\vert_{B_{\delta}}\colon B_{\delta}\to B_{\delta}$ is a contraction, where $B_{\delta}$ is given by $B_{\delta}=\{ u\in C^{4,\alpha}(\ol{\Om}): \norm{u}_{C^{4,\alpha}(\ol{\Om})}\le \delta \} $. Furthermore, \begin{equation}\label{eq:fixed_point_quad_estimate}
		\norm{T_v(h)}_{C^{4,\alpha}(\ol{\Om})} \lesssim \norm{v+h}_{C^{4,\alpha}(\ol{\Om})}^2, \quad h\in B_{\delta}.
	\end{equation}
	Consequently, there exists a unique $r \in B_{\delta}$ solving the fixed point equation $r = T_v(r)$. The function $r$ is also the unique solution of 
	\begin{equation}\label{eq:fixed_point_equation}
		\begin{aligned}
			P_Qr&= -R(v+r) \quad\text{in }\Om, \quad \mbox{and} \quad 
			(r,\Delta r)&=0 \quad\text{on }\p\Om,
		\end{aligned}
	\end{equation}
	where $P_Q r= \Delta^2 r + \partial_{q} Q(x,w,\nabla w,\Delta w) \Delta r + \nabla_p Q(x,w,\nabla w,\Delta w)\cdot\nabla r + \p_uQ(x,w,\nabla w,\Delta w)r.$
\end{lemma}
\begin{proof}
	We start with proving the fact that $T_v$ is a mapping from $B_{\delta}$ into itself for $\delta>0$ small enough. Let $v,r\in B_{\delta} $ for $\delta\le 1$. By Lemma \ref{lem_regularity} we obtain
	\begin{align}
		\norm{T_{v}(r)}_{C^{4,\alpha}(\ol{\Om})}= \norm{G_{Q}(R_v(r),0,0}_{C^{4,\alpha}(\ol{\Om})}&\le C \norm{R(v+r)}_{C^{\alpha}(\ol{\Om})},
	\end{align}
	and by the fundamental theorem of calculus we obtain
	\begin{align*}
		&\norm{R(v+r)}_{C^{\alpha}(\ol{\Om})} \\
		= & \Bigg\lVert\int_0^1 [\p_u Q(x,w+t(v+r),\nabla(w+t(v+r)),\Delta(w+t(v+r))) - \p_u Q(x,w,\nabla w,\Delta w)](v+r) \,\D t \\
		&+ \int_0^1 [\nabla_p Q(x,w+t(v+r),\nabla(w+t(v+r)),\Delta(w+t(v+r))) - \nabla_p Q(x,w,\nabla w,\Delta w)]\cdot \nabla (v+r) \,\D t\\\notag
		&+ \int_0^1 [\p_q Q(x,w+t(v+r),\nabla(w+t(v+r)),\Delta(w+t(v+r))) - \p_q Q(x,w,\nabla w,\Delta w)] \Delta (v+r) \,\D t\Bigg\rVert_{C^{\alpha}(\ol{\Om})}\\
		&\le \Bigg\lVert\int_0^1\int_{0}^1 \frac{d}{d s}\p_u Q(x,w+st(v+r),\nabla(w+st(v+r)),\Delta(w+st(v+r))) (v+r) \dst \Bigg\rVert_{C^{\alpha}(\ol{\Om})} \\&\quad+\Bigg\lVert\int_0^1\int_{0}^1 \frac{d}{d s}\nabla_p Q(x,w+st(v+r),\nabla(w+st(v+r)),\Delta(w+st(v+r))) \cdot \nabla(v+r) \dst \Bigg\rVert_{C^{\alpha}(\ol{\Om})}\\& \quad+ \Bigg\lVert\int_0^1\int_{0}^1 \frac{d}{d s}\p_q Q(x,w+st(v+r),\nabla(w+st(v+r)),\Delta(w+st(v+r))) \Delta(v+r) \dst \Bigg\rVert_{C^{\alpha}(\ol{\Om})}\\
		&\lesssim\Bigg\lVert\int_0^1\int_{0}^1 \p^2_u Q(x,w+st(v+r),\nabla(w+st(v+r)),\Delta(w+st(v+r))) t (v+r)^2 \dst \Bigg\rVert_{C^{\alpha}(\ol{\Om})}\\&+ \Bigg\lVert\int_0^1\int_{0}^1 \p_u\nabla_p Q(x,w+st(v+r),\nabla(w+st(v+r)),\Delta(w+st(v+r))) \cdot t\,\nabla(v+r)(v+r) \dst \Bigg\rVert_{C^{\alpha}(\ol{\Om})}\\&+ \Bigg\lVert\int_0^1\int_{0}^1 \p_q\p_u Q(x,w+st(v+r),\nabla(w+st(v+r)),\Delta(w+st(v+r))) \cdot t\, \Delta(v+r)(v+r) \dst \Bigg\rVert_{C^{\alpha}(\ol{\Om})}\\
		&+\Bigg\lVert\int_0^1\int_{0}^1 \nabla^2_{p} Q(x,w+st(v+r),\nabla(w+st(v+r)),\Delta(w+st(v+r))) : t\, \nabla(v+r)\otimes \nabla(v+r) \dst \Bigg\rVert_{C^{\alpha}(\ol{\Om})}\\
		&+ \Bigg\lVert\int_0^1\int_{0}^1 \p_q\p_q Q(x,w+st(v+r),\nabla(w+st(v+r)),\Delta(w+st(v+r))) \cdot t\, (\Delta(v+r))^2 \dst \Bigg\rVert_{C^{\alpha}(\ol{\Om})}\\
		&+ \Bigg\lVert\int_0^1\int_{0}^1 \p_q\nabla_p Q(x,w+st(v+r),\nabla(w+st(v+r)),\Delta(w+st(v+r))) \cdot t\, \cdot \nabla(v+r)\, \Delta(v+r) \dst \Bigg\rVert_{C^{\alpha}(\ol{\Om})}.
	\end{align*}
	Next using Lemma \ref{Q_deriv_estimate}, we observe that, $\norm{R(v+r)}_{C^{\alpha}(\ol{\Om})} \le C \norm{(v+r)}_{C^{2,\alpha}(\ol{\Om})}^2 \le C \norm{(v+r)}_{C^{4,\alpha}(\ol{\Om})}^2 \le C \delta^2.$
	This shows that $T_v$ is a mapping from $B_{\delta}$ into itself for small $\delta$. The second inequality above proves \eqref{eq:fixed_point_quad_estimate}.
	
	We next prove that $T_v$ is indeed a contraction mapping. To this end, let $r_1,r_2\in B_{\delta}$. Then utilizing the definition $T_v$, we obtain that $T_v(r_1)-T_v(r_2):= G_{Q}(R_v(r_2),0,0)- G_{Q}(R_v(r_1),0,0)$, and it satisfies
	\begin{align}
		\norm{T_v(r_1)-T_v(r_2)}_{C^{4,\alpha}(\ol{\Om})}\le C \norm{R_v(r_1)-R_v(r_2)}_{C^{\alpha}(\ol{\Om})}.
	\end{align}

	We now focus on estimating $C^{\alpha}(\ol{\Om})$-norm of  $R_v(r_1) -R_v(r_2)$. To this end, using Lemma \ref{lem_R_v} we deduce that  
	\begin{align*}
		\norm{R_v(r_1) -R_v(r_2)}_{C^{\alpha}(\ol{\Om})}\le C \norm{u_1-u_2}_{C^{2,\alpha}(\ol{\Om})}(C_1 \delta t + C_2 \delta^2\, t^2).
	\end{align*}
	Next choosing $\delta$ small enough we see that 
	\begin{align*}
		\norm{T_v(r_1) -T_v(r_2)}_{C^{4,\alpha}(\ol{\Om})}\le&   \norm{R_v(r_1) -R_v(r_2)}_{C^{\alpha}(\ol{\Om})} \le \frac{1}{2
		} \norm{u_1-u_2}_{C^{2,\alpha}(\ol{\Om})} \le \frac{1}{2
		} \norm{r_1-r_2}_{C^{4,\alpha}(\ol{\Om})}.
	\end{align*}
	This shows that $T_v$ is a contraction map, and by the Banach fixed point theorem we have that the equation $r=T_v(r)$ has a unique solution in $B_{\delta}$.
\end{proof}
The following lemma provides solutions to the nonlinear equation in the vicinity of the fixed solution \( w \). Our goal is to construct a well-defined solution map \( S_{Q,w} \), where
\[
S_{Q,w}: \text{small solutions of the linearized problem} \longrightarrow \text{solutions of the nonlinear equation close to } w.
\]
More precisely, we establish the following result:
\begin{lemma} \label{lemma_perturbed_sol}
	Let \( Q \in C^k(\Rb,C^{4,\alpha}(\ol{\Omega}), C^{4,\alpha}(\ol{\Omega};\Rnn),C^{4,\alpha}(\ol{\Omega})) \). 
	Let \( w \in C^{k,\alpha}(\overline{\Omega}) \) be a solution of 
	\[
	\Delta^2 w + Q(x, w, \nabla w, \Delta w) = 0.
	\]
	Define the coefficients 
	\begin{align*}
		&	A(x) = \partial_q Q(x, w(x), \nabla w(x), \Delta w(x)), \quad
		X(x) = \nabla_p Q(x, w(x), \nabla w(x), \Delta w(x)), \\
		&	V(x) = \partial_u Q(x, w(x), \nabla w(x), \Delta w(x)).
	\end{align*}
	
	Then the following holds:
	\begin{itemize}
		\item[1.] There exist constants \( \delta, C > 0 \) and a \( C^{k-1} \) map \( \Phi = \Phi_{Q,w}: B_{\delta} \to B_{\delta} \) satisfying 
		\[
		\Phi(B_{\delta})|_{\partial \Omega} = 0, \quad \text{and} \quad \Phi(0) = D\Phi(0) = 0.
		\]
		Moreover, the following estimate holds:
		\begin{equation} \label{lemma_estimate_r}
			\|\Phi(v)\|_{C^{4,\alpha}(\overline{\Omega})} \leq C \|v\|_{C^{4,\alpha}(\overline{\Omega})}^2,
		\end{equation}
		and the map \( S_{Q} = S_{Q,w}: B_{\delta} \to C^{4,\alpha}(\overline{\Omega}) \), defined by 
		$ u = S_{Q,w}(v) = w + v + \Phi(v)$\\
		is a \( C^{k-1} \) map satisfying 
		\begin{equation} \label{u_v_eq}
			\Delta^2 u + Q(x, u, \nabla u, \Delta u) = \Delta^2 v + A \Delta v + X \cdot \nabla v + V v \quad \text{in } \Omega,
		\end{equation}
		with \( S_{Q,w}'(0)v = v \). In particular, if \( v \) solves 
		$
		\Delta^2 v + A \Delta v + X \cdot \nabla v + V v = 0,$\\
		then \( u = S_{Q,w}(v) \) solves 
		$
		\Delta^2 u + Q(x, u, \nabla u, \Delta u) = 0.$
		
		\item[2.] Conversely, if \( \delta \) is sufficiently small, then for any solution \( u \in C^{4,\alpha}(\overline{\Omega}) \) of 
		\[
		\Delta^2 u + Q(x, u, \nabla u, \Delta u) = 0
		\]
		with \( \|u - w\|_{C^{4,\alpha}(\overline{\Omega})} \leq \delta \), there exists a unique solution \( v \in C^{2,\alpha}(\overline{\Omega}) \) of 
		\[
		\Delta^2 v + A \Delta v + X \cdot \nabla v + V v = 0
		\]
		such that \( u = S_{Q,w}(v) \). The function \( v \) is explicitly given by 
		\begin{equation} \label{u_v_converse}
			v = G_Q(0, (u - w)|_{\partial \Omega}, \Delta(u - w)|_{\partial \Omega}),
		\end{equation}
		and satisfies the estimate 
		\[
		\|v\|_{C^{4,\alpha}(\overline{\Omega})}  \lesssim \left( \|u - w\|_{C^{4,\alpha}(\overline{\Omega})} + \|\Delta(u - w)\|_{C^{4,\alpha}(\overline{\Omega})} \right).
		\]
	\end{itemize}
\end{lemma}

\begin{proof}
	The goal is to construct the map $S_Q$. We look for the solution $u$ of \eqref{u_v_eq} in the form\\$u= w+v+r$ and formulate a fixed point equation for $r$. A use of Taylor series expansion implies
	\begin{align*}
		\Delta^2 u +  Q(x,u ,\nabla u, \Delta u)&= \Delta^2 (w+v+r) +  Q(x,w+v+r ,\nabla (w+v+r), \Delta (w+v+r))\\
		&= \Delta^2 w+ \Delta^2(v+r) +  Q(x,w ,\nabla w, \Delta w)\\& \qquad+ [A\Delta  + X\cdot\nabla  + V](v+r)+R_v(r), 
	\end{align*}
	where $R_v(r)= R(v+r)$ is given by 
	\begin{align*}
		R(h) &:= \int_0^1 [\p_u Q(x,w+th,\nabla(w+th),\Delta(w+th)) - \p_u Q(x,w,\nabla w,\Delta w)]h \,\D t \\\notag
		&\qquad+ \int_0^1 [\nabla_p Q(x,w+th,\nabla(w+th),\Delta(w+th)) - \nabla_p Q(x,w,\nabla w,\Delta w)]\cdot \nabla h \,\D t\\\notag
		&\qquad+ \int_0^1 [\p_q Q(x,w+th,\nabla(w+th),\Delta(w+th)) - \p_q Q(x,w,\nabla w,\Delta w)]\,\Delta h \,\D t.
	\end{align*}
	Since $w$ is a solution of $  \Delta^2 u +  Q(x,u ,\nabla u, \Delta u)=0$, this implies $u$ solves 
	\eqref{u_v_eq} if $r$ satisfies
	\begin{align*}
		\Delta^2 r + A \Delta r + X \cdot \nabla r +V r= - R_v(r) \quad \mbox{in $\Omega$}.
	\end{align*}
	By Lemma \ref{lemma_fixed_point}, we have that for each $v\in B_{\delta}$ there is a unique $r=r_v$ satisfying the above equation. Hence the mapping $ v\mapsto r_v$ is well-posed for $v \in B_{\delta}$. Next we show that this map is indeed a $C^{k-1}$ map using implicit function theorem. To this end, define the map
	\begin{align}\label{def_of_F}
		F: C^{4, \alpha}(\ol{\Om}) \times C^{4, \alpha}(\ol{\Om}) \rightarrow C^{4, \alpha}(\ol{\Om}) \quad \mbox{by} \quad F(v,r)= r- T_v(r)= r+ G_{Q}(R_v (r),0,0).
	\end{align}
	Note that $F$ satisfies $ F(0,0)=0$. Next we see that  $F\in C^{k-1}$, since $R_v\in C^{k-1}$ and $G_Q$ is linear. Moreover, $( D_rF)|_{(0,0)}(h)=h.$
	This implies it is a linear isomorphism from $ C^{4, \alpha}(\ol{\Om})$ into itself. Hence by the implicit function theorem, there exist open balls $B_{\delta_1}$ and $B_{\delta_2}$ and a $C^{k-1}$ map $\Phi: B_{\delta_1}\rightarrow B_{\delta_2}$ such that 
	\begin{align*}
		F(v, \Phi(v))=0.
	\end{align*}
	Note that $r_v$ found by Lemma \ref{lemma_fixed_point} is the unique solution of $F(v, \cdot)=0$ in $B_{\delta}$ for $ v\in B_{\delta}$. Thus we conclude  that for $\delta< \min \{\delta_1, \delta_2\}$, $r_v \in B_{\delta_2} $, and $\Phi(v)= r_v
	$. Consequently, we have show that for each $v\in B_{\delta}$ there is a unique $r_v\in B_{\delta}$ such that $u=w+v+r_v$ is a solution of \eqref{u_v_eq}. Moreover, the map
	\begin{align*}
		S_{Q,w}: B_{\delta}\rightarrow C^{4,\alpha}(\ol{\Om}) \quad \mbox{ given by} \quad  \mbox{$  u_v=S_{Q,w}(v)$ is $C^{k-1}$}. 
	\end{align*}
	We next show that $\Phi$ satisfies the other properties. By \eqref{eq:fixed_point_quad_estimate} we have that 
	\begin{align*}
		\norm{r}_{C^{4,\alpha}(\ol{\Om})}= \norm{T_v(r)}_{C^{4,\alpha}(\ol{\Om})}\le C \norm{(v+r)}^2_{C^{4,\alpha}(\ol{\Om})}\le C \left(\norm{r}^2_{C^{4,\alpha}(\ol{\Om})}+ \norm{v}^2_{C^{4,\alpha}(\ol{\Om})}\right).
	\end{align*}
	This further entails $\norm{v}^2_{C^{4,\alpha}(\ol{\Om})}\ge  \norm{r}_{C^{4,\alpha}(\ol{\Om})} (1-C \delta)$, since $\norm{r}_{C^{4,\alpha}(\ol{\Om})}\le \delta$.
	Next using the fact that $\Phi(v)=r$ and $\delta$ small enough we conclude that
	\begin{align*}
		\norm{\Phi(v)}_{C^{4,\alpha}(\ol{\Om})}   \le C\norm{v}^2_{C^{4,\alpha}(\ol{\Om})}, 
	\end{align*}
	which in turn implies \eqref{lemma_estimate_r}, and $\Phi(0)=0$. The preceding estimate further gives that $ D \Phi(0)=0$. This completes the proof of first part. 
	
	We now focus on the converse statement.
	Suppose $u$ solves $ \Delta^2 u + Q(x, u, \nabla u , \Delta u)=0$ in $ \Omega$ and $ \norm{u-w}_{C^{4,\alpha}(\ol{\Om})}\le \delta $. Denote $ \tilde{u}= u-w$ and we wish to construct $v$ solving the linear partial differential equation $ (\Delta^2+ A \Delta + X\cdot \nabla +V ) v=0$ in $\Omega$ such that $ \tilde{ u}= v+ \Phi (v)$. Let $ \phi$ be the unique solution of \eqref{eq:boundary_value_problem} given by Lemma \ref{lem_regularity} with $F=0$ and $(f_0,f_1)=( \tilde{u}, \Delta \tilde{u})|_{\p\Om}$. Denote $v=\phi$,  this gives
	\begin{align*}
		\norm{v}_{C^{4,\alpha}(\ol{\Om})}   \le C \left( \norm{\tilde{u}}_{C^{4,\alpha}(\ol{\Om})} + \norm{\Delta \tilde{u}}_{C^{2,\alpha}(\ol{\Om})}  \right).
	\end{align*}
	It remains to show that $ r= u-w-v$ satisfies $r=\Phi(v)$. To this end, we observe that 
	\begin{align*}
		&
		(\Delta^2  + A\Delta  + X\cdot\nabla  + V)r \\&=   (\Delta^2  + A\Delta  + X\cdot\nabla  + V)(u-w)\\
		&= Q ( x, w, \nabla w, \Delta w)- Q ( x, u, \nabla u, \Delta u) - (  A\Delta  + X\cdot\nabla  + V)(u-w) \\
		&= Q ( x, w, \nabla w, \Delta w)- Q ( x, (w+v+r), \nabla (w+v+r), \Delta (w+v+r)) - (  A\Delta  + X\cdot\nabla  + V)(v+r)\\
		&=- R_v(r).
	\end{align*}
	The first part of the proof ensures that $r= \Phi(v)$ if $\delta$ is chosen small enough. This proves that $u=S_{Q,w}(v)$. To show the uniqueness of $v$ we assume that $u=S_{Q,w}(v_1)=S_{Q,w}(v_2)$. Then by definition we have $ v_1-v_2= r_1-r_2= \Phi(v_1)-\Phi(v_2)$. Since $ (v_1-v_2)$ solves 
	\begin{equation*}
		\begin{aligned}
			(\Delta^2+ A \Delta + X\cdot \nabla +V) u&=0 
			\quad \mbox{in $ \Omega$}\\
			( u, \Delta u)=(0,0) \quad \mbox{on $ \p\Om$}.
		\end{aligned}
	\end{equation*}
	Hence by Lemma \ref{lem_regularity} we have that $v_1=v_2$, proving the uniqueness of $v$. The proof is complete.
\end{proof}
To state our next result we recall that $	\V_{A,X,V}= \{ u\in {C^{4,\alpha}(\ol{\Om})}: (\Delta^2+ A \Delta + X\cdot \nabla +V) u=0\}.$
\begin{lemma}\label{lem_differomorphism}
	In the setting of Lemma \ref{lemma_perturbed_sol}, if $v$ is small and solves $ (\Delta^2+ A \Delta + X\cdot \nabla +V) u=0$ for $A=\p_qQ(x,w,\nabla w,\Delta w)$, $X=\nabla_pQ(x,w,\nabla w,\Delta w)$ and $V=\p_qQ(x,w,\nabla w,\Delta w)$. Define
	\begin{align*}
		A_v(x) &:= \p_qQ(x, S_{Q,w}(v),\nabla S_{Q,w}(v),\Delta S_{Q,w}(v)),\\
		X_v(x)& := \nabla_pQ(x, S_{Q,w}(v),\nabla S_{Q,w}(v),\Delta S_{Q,w}(v)),\\
		V_v(x)& := \p_u Q(x, S_{Q,w}(v),\nabla S_{Q,w}(v),\Delta S_{Q,w}(v)).
	\end{align*}
	If $v\in \V_{A,X,V}$ is small, then the map $ DS_{Q,w}$ is an isomorphism from $\V_{A,X,V}$ onto $ \V_{A_v,X_v,V_v}$.
\end{lemma}
\begin{proof}
	Define  $v_t= v+th$ for $t$ small and $h\in \V_{A,X,V} $. Then $u_t= S_{Q,w}(v_t)$ solves  $ \Delta^2 u + Q( x, u, \nabla u,\Delta u)=0$ in $ \Om$. Since $u_t$ is smooth in $t$, the function $\dot{u}_0= \p_t u_t|_{t=0}= DS_{Q,w}(v)h$ satisfies 
	\begin{align*}
		(\Delta^2  + A_v \Delta + X_v(x)\cdot \nabla + V_v(x))\dot{u}_0=0.
	\end{align*}
	This shows that $DS_{Q,w}$ is a mapping from $\V_{A,X,V}$ into $\V_{A_v,X_v,V_v}$.
	
	Suppose $v\in V_{A,X,V} $ is small. For $t$ small define $u_t= S_{Q,S_{Q,w}(v)}(t\tilde{h})$, then by the converse part of the Lemma \ref{lemma_perturbed_sol}, we have $ u_t= S_{Q,w}(v_t)$ for a unique small solution $v_t\in  \V_{A,X,V} $ and it is given by 
	\begin{align*}
		v_t= G_{Q}(0, (u_t-w)|_{\p\Om}, \Delta (u_t-w)|_{\p\Om}).
	\end{align*}
	Since $S_{Q,w}(v_0) = u_0= S_{Q,S_{Q,w}(v)}(0)= S_{Q,w}(v) $, this implies $v=v_0$. Furthermore, differentiating the relations $ u_t= S_{Q,w}(v_t) $ and $u_t=S_{Q,S_{Q,w}(v)}(t\tilde{h}) $ with respect to $t$ we obtain
	\begin{align*}
		DS_{Q,w}(v)\dot{v}_0= \dot{u}_0= DS_{Q,S_{Q,w}(v)}(0)\tilde{h}= \tilde{h}
	\end{align*}
	This shows that $ DS_{Q,w}: \V_{A,X,V} \rightarrow \V_{A_v,X_v,V_v}$ surjective. 
	
	Next we show that for $h \in \V_{A,X,V}  $ such that $DS_{Q,w}(v) h=0$, then $h=0$. Since $S_{Q,w} (v)= w+v+ \Phi(v)$, this together with $DS_{Q,w}(v) h=0$ implies  $h + D\Phi(v)h=0$. However, $D\Phi(0)=0$ together with $Q$ being at least $C^2$ entails $ \norm{D\Phi(v)} \le \frac{1}{2}$, when $v $ is small enough. This implies $ \norm{h}\le \frac{1}{2} \norm{h} \implies h=0$. Thus $DS_{Q,w}$ is bijective and bounded, and by open mapping theorem it is an isomorphism. This completes the proof.
\end{proof}

\section{Estimates for solutions based on their Cauchy data}\label{sec:estimating_Cauchy_data}

We follow the method presented in \cite{JNS2023}. The proofs of the following Lemmas are similar to the ones in \cite{JNS2023}*{Section 3} but we record them here for completeness.

Let $u\in C^{4,\alpha}(\ol{\Omega})$, then we define three operators $B_{\p\Om,\mathcal{N}}\colon C^{4,\alpha}(\ol{\Om})\to C^{4,\alpha}(\p\Om)\times C^{3,\alpha}(\p\Om)\times C^{2,\alpha}(\p\Om)\times C^{1,\alpha}(\p\Om),$ $B_{\p\Om,\mathcal{N}}^N\colon C^{4,\alpha}(\p\Omega) \rightarrow \Rb$ and $L\colon C^{4,\alpha}(\Omega) \rightarrow C^{\alpha}(\Omega) $ by 
\begin{align}
	B_{\p\Om,\mathcal{N}}(u)&:= (u|_{\p\Om},\p_{\nu}u|_{\p\Om}, \Delta u|_{\p\Om}, \p_{\nu}\Delta u|_{\p\Om}), \\
	B_{\p\Om,\mathcal{N}}^N(u) &:=\norm{u}_{C^{4,\alpha}(\p\Omega)}+ \norm{\p_{\nu}u}_{C^{3,\alpha}(\p\Omega)} +\norm{\Delta u}_{C^{2,\alpha}(\p\Omega)} +\norm{\p_{\nu}\Delta u}_{C^{1,\alpha}(\p\Omega)},\\
	Lu&:= (\Delta^2+ A(x)\Delta + X(x)\cdot \nabla + V(x))u.
\end{align}
Here $V,A\in C^{k,\alpha}(\ol{\Om})$ and $X\in C^{k,\alpha}(\ol{\Om};\Rnn)$.

\begin{lemma}\label{Lemma_est_1}
	Let $\Omega\subset \Rnn$ be a bounded domain with smooth boundary. Then there is a constant 
	$C>0$ such that for any $u\in C^{4,\alpha}(\ol{\Omega})$ one has
	\begin{align}
		\norm{u}_{C^{4,\alpha}(\ol{\Omega})}\lesssim \left(\norm{u}_{C^{4,\alpha}(\p\Omega)}+ \norm{Lu}_{C^{\alpha}(\ol{\Omega})} + \norm{u}_{H^2(\Omega)} \right).
	\end{align}
	
\end{lemma}
\begin{proof}
	We first define the Banach space $X=C^{4,\alpha}(\p\Om)\times C^{\alpha}(\ol{\Om})\times H^{2}(\Om)$, with
	\begin{equation*}
		\norm{(f,F,u)}_X=\norm{f}_{C^{4,\alpha}(\p\Om)}+\norm{F}_{C^{\alpha}(\ol{\Omega})}+\norm{u}_{H^2(\Om)},
	\end{equation*}
	and an operator $T\colon C^{4,\alpha}(\ol{\Omega})\to X$,
	\begin{equation*}
		T(u)=(u|_{\p\Om},Lu,j(u)).
	\end{equation*}
	Here $j\colon C^{4,\alpha}(\ol{\Omega})\to H^2(\Om)$ is the natural inclusion operator. Now $T$ is continuous, injective and linear. If we would also have that $T$ is bijective onto $\mathrm{Ran}(T)$ then by the open mapping theorem we would have a bounded inverse $S\colon \mathrm{Ran}(T)\to C^{2,\alpha}(\ol{\Om})$ and thus
	\begin{equation*}
		\norm{u}_{C^{2,\alpha}(\ol{\Om})}=\norm{STu}_{C^{2,\alpha}(\ol{\Om})}\leq C\norm{Tu}_{X}
	\end{equation*}
	for any $u\in C^{2,\alpha}(\ol{\Om})$, which would prove the claim.
	
	To show bijectivity of $T$ we need to prove that $\mathrm{Ran}(T)$ is closed. Let $u_j\in C^{4,\alpha}(\ol{\Om})$ such that $T(u_j)\to (f,F,u)$ in $X$. Then $u_j|_{\p\Om}\to f$ in $C^{4,\alpha}(\p\Om)$, $Lu_j\to F$ in $C^{\alpha}(\ol{\Om})$ and $u_j\to u$ in $H^2(\Om)$. On the other hand also $u_j|_{\p\Om}\to u|_{\p\Om}$ in $H^{1/2}(\p\Om)$, $Lu_j\to Lu$ in $H^{-2}(\Om)$ and thus by uniqueness of limits we get $Lu=F$ and $u_{\p\Om}=f$. Then by elliptic regularity $u\in C^{4,\alpha}(\ol{\Om})$ and thus $T(u)=(f,F,u)$ which shows that $\mathrm{Ran}(T)$ is closed.
\end{proof}

\begin{lemma}\label{Lemma_est_2}
	Let $\Omega\subset \Rnn$ be a bounded domain with smooth boundary. Then there is a constant 
	$C>0$ such that for any $u\in C^{4,\alpha}(\ol{\Omega})$ one has
	\begin{align}
		\norm{u}_{C^{4,\alpha}(\ol{\Om})}\lesssim \left(B_{\p\Om,\mathcal{N}}^N(u)+ \norm{Lu}_{C^{\alpha}(\ol{\Om})}  \right).
	\end{align}
\end{lemma}

\begin{proof}
	We argue by contradiction, thus suppose that for any $m$ there is $u_m\in C^{4,\alpha}(\ol{\Om})$ such that
	\begin{equation}\label{eq_contradiction}
		\norm{u_m}_{C^{4,\alpha}(\ol{\Om})} > m \left(B_{\p\Om,\mathcal{N}}^N(u_m)+ \norm{Lu_m}_{C^{\alpha}(\ol{\Om})}  \right).
	\end{equation}
	By Lemma \ref{Lemma_est_1} we have $ 	\norm{u_m}_{C^{4,\alpha}(\ol{\Om})}\le C \left(\norm{u_m}_{C^{4,\alpha}(\p\Om)} + \norm{Lu_m}_{C^{\alpha}(\ol{\Om})} + \norm{u_m}_{H^2(\Om)} \right)$, and if we scale $u_m$ so that $\norm{u_m}_{H^2(\Om)}=1$ we obtain $\norm{u_m}_{C^{4,\alpha}(\ol{\Om})}\le C \left(\norm{u_m}_{C^{4,\alpha}(\p\Om)} + \norm{Lu_m}_{C^{\alpha}(\ol{\Om})} + 1 \right).$
A 	combination of  this with \eqref{eq_contradiction} entails
	\begin{align}\label{eq_uni_bound}
		\norm{u_m}_{C^{4,\alpha}(\ol{\Om})}&\le C\left( B_{\p\Om,\mathcal{N}}^N(u_m) + \norm{Lu_m}_{C^{\alpha}(\ol{\Om})} + 1 \right)\le C\left(\frac{1}{m}\norm{u_m}_{C^{4,\alpha}(\ol{\Om})} + 1 \right),
	\end{align}
	and hence the sequence $u_m$ is uniformly bounded in $C^{4,\alpha}(\ol{\Om})$ for $m$ sufficiently large.
	
	By \cite{Adams_sobolev}*{Theorem 1.34} there is a compact embedding from $C^{4,\alpha}(\ol{\Om})$ to $C^{4}(\ol{\Om})$ and thus there is converging subsequence of $u_m$, denoted by $u_m$, such that $u_m\to u\in C^{4}(\ol{\Om})$. For this subsequence we also get from \eqref{eq_contradiction}, \eqref{eq_uni_bound} that
	\begin{equation*}
		Lu_m\to 0,\quad B_{\p\Om,\mathcal{N}}(u_m)\to 0
	\end{equation*}
	in the corresponding spaces. This implies that $B_{\p\Om,\mathcal{N}}(u)=0$ and $Lu=0$ by the uniqueness of limits. Now by unique continuation we must have $u=0$. But this is a contradiction since $\norm{u}_{H^2(\Om)}=\lim_{m\to\infty}\norm{u_m}_{H^2(\Om)}=1$.
\end{proof}

\begin{lemma}\label{lemma_cauchy_data_est}
	Let $Q \in C^k(\Rb,C^{4,\alpha}(\ol{\Omega}), C^{4,\alpha}(\ol{\Omega};\Rnn),C^{4,\alpha}(\ol{\Omega}))$ and $u_1\in C^{4,\alpha}(\ol{\Omega})$ satisfy $ \Delta^2 u_1 +  Q(x,u_1 ,\nabla u_1, \Delta u_1)=0 $ in $\Om$. If $u_2\in C^{4,\alpha}(\ol{\Omega})$  is any other solution of $\Delta^2 u +  Q(x,u ,\nabla u, \Delta u)=0$ in $\Om$ and $\norm{u_i}_{C^{4,\alpha}(\ol{\Omega})}\le M$ for $i=1,2$, then 
	\begin{align}
		\norm{u_1-u_2}_{C^{4,\alpha}(\ol{\Omega})} \le C(M,Q) \cdot B_{\p\Om,\mathcal{N}}^N(u_1-u_2).
	\end{align}
\end{lemma}

\begin{proof}
	Denote $v=u_1-u_2$. Then from Lemma \ref{Lemma_est_1} we get
	\begin{align}\label{lem_cauchy_data_est_1}
		\norm{v}_{C^{4,\alpha}(\ol{\Om})}&\le C \left(\norm{v}_{C^{4,\alpha}(\p\Omega)}+ \norm{Lv}_{C^{\alpha}(\ol{\Omega})} + \norm{v}_{H^2(\Omega)} \right)\\\notag
		&\leq C \left( \norm{v}_{C^{4,\alpha}(\p\Omega)}+ \norm{\Delta^2v}_{C^{\alpha}(\ol{\Omega})} + \norm{v}_{H^2(\Omega)} \right).
	\end{align}
	Hence we still need to estimate $\norm{\Delta^2v}_{C^{\alpha}(\ol{\Om})}$ and $\norm{v}_{H^2(\Omega)}$. Now if we denote $v_t=tu_2+(1-t)u_1$, then
	\begin{align}\label{taylor_expansion}\notag
		-\Delta^2v &= Q(x,u_2,\nabla u_2,\Delta u_2) - Q(x,u_1,\nabla u_1,\Delta u_1)=\int_0^1\frac{d}{dt}Q(x,v_t,\nabla v_t, \Delta v_t)\, \D t\\
		&=\int_0^1\p_uQ(x,v_t,\nabla v_t, \Delta v_t)v + \nabla_pQ(x,v_t,\nabla v_t, \Delta v_t)\cdot\nabla v + \p_qQ(x,v_t,\nabla v_t, \Delta v_t)\Delta v \, \D t.
	\end{align}
	Thus we conclude from Lemma \ref{Q_deriv_estimate}
	\begin{align*}\label{est_bilaplace}
		\norm{\Delta^2v}_{C^{\alpha}(\ol{\Om})}\leq C(M,Q)\left(\norm{v}_{C^{\alpha}(\ol{\Om})} + \norm{\nabla v}_{C^{\alpha}(\ol{\Om})} + \norm{\Delta v}_{C^{\alpha}(\ol{\Om})} \right).
	\end{align*}
	Notice that here we use the assumption $\norm{u_i}_{C^{4,\alpha}(\ol{\Omega})}\le M$.
	Applying Sobolev embeddings, interpolation of $L^p$-spaces and Young's inequality with $\e$ (as in the proof of \cite{JNS2023}*{Lemma 3.3}; the steps are the same) we get
	\begin{equation*}
		\norm{v}_{C^{\alpha}(\ol{\Om})}\leq C \norm{v}_{L^2(\Om)}, \quad \norm{\nabla v}_{C^{\alpha}(\ol{\Om})}\leq C \norm{\nabla v}_{L^2(\Om)}, \quad \norm{\Delta v}_{C^{\alpha}(\ol{\Om})}\leq C \norm{\Delta v}_{L^2(\Om)}.
	\end{equation*}
	Combining this with \eqref{lem_cauchy_data_est_1} implies
	\begin{equation}\label{lem_cauchy_data_est_2}
		\norm{v}_{C^{4,\alpha}(\ol{\Om})}\leq C \left( \norm{v}_{C^{4,\alpha}(\p\Omega)} + \norm{v}_{H^2(\Omega)} \right).
	\end{equation}
	To estimate $\norm{v}_{H^2(\Omega)}$ we use a Carleman estimate (see e.g. \cite{ChoulliPDEA2021}*{Theorem 4.1.}): there are $C, \tau_0 > 0$ and $\varphi \in C^{\infty}(\ol{\Om})$ such that when $\tau \geq \tau_0$, one has 
	\begin{equation}\label{carleman_est}
		\norm{e^{\tau \varphi} u}_{L^2(\Om)} + \frac{1}{\tau} \norm{e^{\tau \varphi} \nabla u}_{L^2(\Om)}\leq \frac{C}{\tau^{3/2}} \norm{e^{\tau \varphi} \Delta u}_{L^2(\Om)} + C \norm{e^{\tau \varphi} u}_{L^2(\p \Om)}+ \frac{C}{\tau} \norm{e^{\tau \varphi} \nabla u}_{L^2(\p \Om)}
	\end{equation}
	for any $v \in C^2(\ol{\Om})$. Using only
	\begin{equation*}
		\norm{e^{\tau \varphi} u}_{L^2(\Om)} \leq \frac{C}{\tau^{3/2}} \norm{e^{\tau \varphi} \Delta u}_{L^2(\Om)} + C \norm{e^{\tau \varphi} u}_{L^2(\p \Om)}+ \frac{C}{\tau} \norm{e^{\tau \varphi} \nabla u}_{L^2(\p \Om)}
	\end{equation*}
	to estimate $\frac{1}{\tau^{3/2}}\norm{e^{\tau\varphi}\Delta v}_{L^2(\Om)}$ and combining with \eqref{carleman_est} we have
	\begin{align*}
		&\norm{e^{\tau \varphi} v}_{L^2(\Om)} + \frac{1}{\tau} \norm{e^{\tau \varphi} \nabla v}_{L^2(\Om)} + \frac{1}{\tau^{3/2}}\norm{e^{\tau \varphi} \Delta v}_{L^2(\Om)} \\
		&\leq \frac{C}{\tau^3}\norm{e^{\tau \varphi} \Delta^2v}_{L^2(\Om)} + C\norm{e^{\tau \varphi} v}_{L^2(\p\Om)} + \frac{C}{\tau}\norm{e^{\tau \varphi} \nabla v}_{L^2(\p\Om)} + \frac{C}{\tau^{3/2}}\norm{e^{\tau \varphi} \Delta v}_{L^2(\p\Om)} \\
		&+ \frac{C}{\tau^{5/2}}\norm{e^{\tau \varphi} \nabla (\Delta v)}_{L^2(\p\Om)}.
	\end{align*}
	Using the observation \eqref{taylor_expansion} and choosing $\tau=\tau(M,Q)$ large but fixed gives
	\begin{align*}
		&\tilde{C}\left(\norm{e^{\tau \varphi} v}_{L^2(\Om)} + \norm{e^{\tau \varphi} \nabla v}_{L^2(\Om)} + \norm{e^{\tau \varphi} \Delta v}_{L^2(\Om)} \right)\\
		&\leq  C\left(\norm{e^{\tau \varphi} v}_{L^2(\p\Om)} + \norm{e^{\tau \varphi} \nabla v}_{L^2(\p\Om)} + \norm{e^{\tau \varphi} \Delta v}_{L^2(\p\Om)} + \norm{e^{\tau \varphi} \nabla (\Delta v)}_{L^2(\p\Om)}\right).
	\end{align*}
	Then $c(M,Q)\leq e^{\tau\varphi} \leq C(M,Q)$ implies $	\norm{v}_{H^2(\Om)}\leq C(M,Q)\norm{v}_{H^3(\p\Om)} $. 
This  together with \eqref{lem_cauchy_data_est_2} proves the claim.
\end{proof}

\section{Controlling Cauchy data}\label{sec:controlling_Cauchy_data}

In Section \ref{sec:smooth_solution_map}, we have  constructed a smooth solution map for the equation, $\Delta^2 u + Q(x, u, \nabla u, \Delta u) = 0$ in $\Omega,$ however that is not enough to proceed further. Therefore the goal of this section is construct a smooth solution map $T_{Q_2,w_2}$ for the nonlinear equation $ \Delta^2+ Q_2(x,u, \nabla u, \Delta u)=0$ in $\Om$ parameterized on solutions $v$ of 
\[
\Delta^2 v + A_1\Delta v + X_1\cdot \nabla v + V_1 v=0 \quad
\mbox{in $\Om$},
\]
which will be particularly useful when proving the main results.
There are two primary reasons for constructing such a map:
\begin{itemize}
	\item[1.] Suppose \( u_1 = S_{Q_1, w_1}(v_1) \). To proceed we need to construct a solution $u_2$ for the nonlinear equation with $Q_2$ in such a way that it has the same Cauchy data as $u_1$ and depends smoothly on the solution $v_1$. That is on the solutions of the linearized equation for $Q_1$. The inclusion of Cauchy data gives the first requirement, but the second is not guaranteed by the first solution operator constructed in Section \ref{sec:smooth_solution_map}. 
	
	\item[2.] Another challenge arises when identifying the first-order derivatives, namely 
	\[
	\partial_u Q(x, w_i, \nabla w_i, \Delta w_i), \quad
	\nabla_p Q(x, w_i, \nabla w_i, \Delta w_i), \quad
	\partial_q Q(x, w_i, \nabla w_i, \Delta w_i),
	\]
	for \( i = 1, 2 \). This identification relies on the linearization method, which requires differentiating the solution maps \( S_{Q_i, w_i} \) in the same direction \( v \). However, to use the same parameter \( v \) for both solution operators \( S_{Q_i, w_i} \), \( v \) must satisfy 
	\[
	(\Delta^2 + A_i \Delta + X_i \cdot \nabla + V_i) v = 0 \quad \text{in } \Omega,
	\]
	for \( i = 1, 2 \). This condition is not guaranteed prior to identifying the first-order derivatives.
\end{itemize}

To this end, we define the following function spaces:
\begin{equation*}
	\begin{aligned}
		Y&:= \{ u\in C^{4,\alpha} (\overline{\Omega}): u=\p_{\nu}u=\Delta u=\p_{\nu} \Delta u =0 \,\, \mbox{on $\p\Omega$} \},\\
		Z&:=\{ (\Delta^2+ A \Delta + X\cdot \nabla +V)u: u\in Y\}.
	\end{aligned}
\end{equation*}
\begin{lemma}
	Then the function spaces $Y$ and $Z$ are Banach spaces.
\end{lemma}

\begin{proof}
	It is clear that $Y$ are Banach space, since it is closed. One can also show that $Z$ is a Banach space. Let $u_n=Lv_n\in Z$ be such that $u_n\to u\in C^{\alpha}(\ol{\Om})$. Since $B_{\p\Om,\mathcal{N}}(u_n)=0$ for all $n$, we have by Lemma \ref{Lemma_est_2} that $	\norm{v_n-v_m}_{C^{4,\alpha}(\ol{\Om})}\leq C\norm{Lv_n-Lv_m}_{C^{\alpha}(\ol{\Om})}.$
	Hence $v_n$ is a Cauchy sequence in $Y$ which implies the existence of $v\in Y$ so that $v_n\to v$ in $C^{4,\alpha}(\ol{\Om})$. Now $	\norm{u_n-Lv}_{C^{\alpha}(\ol{\Om})} = \norm{L(v_n-v)}_{C^{\alpha}(\ol{\Om})} \leq C\norm{v_n-v}_{C^{4,\alpha}(\ol{\Om})}.$
	Thus $u_n\to u=Lv$ in $C^{\alpha}(\ol{\Om})$ which says that $Z$ is a Banach space.
\end{proof}

The following lemma demonstrates the existence of a bounded projection operator $P: C^{4,\alpha}(\ol{\Om})\rightarrow C^{4,\alpha}(\ol{\Omega})\cap Z$. For this, recall from Section \ref{sec:estimating_Cauchy_data} the operator $L\colon C^{4,\alpha}(\Om) \to C^{\alpha}(\Om),$
\begin{equation*}
	Lu=(\Delta^2+ A(x)\Delta + X(x)\cdot \nabla + V(x))u.
\end{equation*}
We define the formal adjoint of $L$ to be $L^*\colon C^{4,\alpha}(\Om) \to C^{\alpha}(\Om),$
\begin{equation*}
	L^*u=\Delta^2u+ \Delta(A(x)u) - \nabla\cdot(Xu) + V(x)u.
\end{equation*}

\begin{lemma}\label{lemma_projection}
	Let $Y$ and $Z$ be as above. There is a projection operator $P : C^{4,\alpha}(\ol{\Omega})\rightarrow C^{4,\alpha}(\ol{\Omega})\cap Z$ given by  $P(u) = Ly,$
	where $y\in C^{8,\alpha}(\ol{\Om})$ is the unique solution of 
	\begin{equation}\label{projection_bvp}
		\begin{cases}
			L^*(Ly) =   L^*u & \quad\text{in }\Om \\
			y=\p_{\nu}y=\Delta y=\p_{\nu}\Delta y=0&\quad\text{on }\p\Om.
		\end{cases}
	\end{equation}
\end{lemma}

\begin{proof}
	Let $y_1, y_2$ be two solutions for the above boundary value problem and denote $y=y_1-y_2$. Then $L^*(Ly)=0$ and
	\begin{align*}
		0&=\int_{\Om}L^*(Ly)\cdot y\,dx=\int_{\Om}(Ly)^2\,dx
	\end{align*}
	since $y$ satisfies the boundary conditions in \eqref{projection_bvp}. Thus $y_1=y_2$ and solutions to \eqref{projection_bvp} are unique.
	
	Existence of solutions can be seen by using standard arguments as follows. Define a coercive bilinear form $B(u,v)=(Lu,L^*v)_{L^2(\Om)} + \gamma_1(\Delta u,\Delta v)_{L^2(\Om)} + \gamma_2(\nabla u,\nabla v)_{L^2(\Om)} + \gamma_3(u,v)_{L^2(\Om)}$, $u,v\in Y_2,$ for $\gamma_j>0$, $j=1,2,3$, sufficiently large depending on the coefficients of $L$. Using the Riesz representation theorem gives existence of weak solutions in $H^4_0(\Om)$ for $L^*(Ly)+\gamma y = F\in H^{-4}(\Om)$. Now the uniqueness argument above gives that one can solve the equation $L^*(Ly)=F$ uniquely even though Fredholm theory tells us that there is a countable set of eigenvalues where uniqueness of solutions could fail. Then by elliptic regularity for $u\in C^{4,\alpha}(\ol{\Om})$ one has $y\in C^{8,\alpha}(\ol{\Om})$.
	
	We are left to prove that $P$, as defined above, is a projection. Let $u \in C^{4,\alpha}(\ol{\Om})$ and $y \in C^{8,\alpha}(\ol{\Om})$ be the unique solution to \eqref{projection_bvp}. Now
	\begin{equation*}
		P(Pu)=P(Ly)=Lv,
	\end{equation*}
	where $v$ is the unique solution of
	\begin{equation*}
		\begin{cases}
			L^*(Lv) =   L^*(Ly) & \quad\text{in }\Om \\
			v=\p_{\nu}v=\Delta v=\p_{\nu}\Delta v=0&\quad\text{on }\p\Om.
		\end{cases}
	\end{equation*}
	This implies that $w=v-y$ solves
	\begin{equation*}
		\begin{cases}
			L^*(Lw) =  0 & \quad\text{in }\Om \\
			w=\p_{\nu}w=\Delta w=\p_{\nu}\Delta w=0&\quad\text{on }\p\Om
		\end{cases}
	\end{equation*}
	and hence by unique continuation $v=y$. This implies $P(Pu)=Lv=Ly=P(u)$ and thus $P$ is a projection onto $C^{4,\alpha}(\ol{\Omega})\cap Z$.
\end{proof}

\begin{lemma}\label{lem_bdd_inv}
	Let $Y$ and $Z$ be as defined above. Then the operator \[\Delta^2+ A(x) \Delta + X\cdot \nabla +V: Y \rightarrow Z \] is bounded and bijective, and has a bounded inverse $G: Z \rightarrow Y$.
\end{lemma}

\begin{proof}
	By definition of $Z$, we have that $ \Delta^2+ A(x) \Delta + X\cdot \nabla +V $ is surjective. To show the injectivity, we assume that there is $y\in Y$ and $\tilde{y}\in Y$ such that 
	\begin{align*}
		( \Delta^2+ A(x) \Delta + X\cdot \nabla +V) y =  ( \Delta^2+ A(x) \Delta + X\cdot \nabla +V) \tilde{y}.
	\end{align*}
	Then the function $u= y-\tilde{y}$ solves the following:
	\begin{equation}
		\begin{aligned}
			( \Delta^2+ A(x) \Delta + X\cdot \nabla +V) y&=0 \quad \mbox{in $\Om$},\\
			(u,\Delta u, \p_{\nu} u, \p_{\nu}\Delta u)&=0\quad\mbox{on $ \p\Om$}.
		\end{aligned}
	\end{equation}
	Therefore by unique continuation principle we conclude that $y =\tilde{y}$. This implies the operator $( \Delta^2+ A(x) \Delta + X\cdot \nabla +V) $ is injective. Moreover,  one has
	\begin{align*}
		& \norm{ ( \Delta^2+ A(x) \Delta + X\cdot \nabla +V)u}_{C^{\alpha}(\ol{\Omega})}\le C \norm{u}_{C^{4,\alpha}(\ol{\Omega})}.
	\end{align*}
	This says that $ \Delta^2+ A(x) \Delta + X\cdot \nabla +V$ is a bounded linear operator which is bijective. Hence, by open mapping theorem there exists a bounded inverse $G:Z\rightarrow Y$ of $\Delta^2+ A(x) \Delta + X\cdot \nabla +V$. This finishes the proof.  
\end{proof}

To proceed further, we now use the ball $N_{A,X,V,\delta}$ in the solution space:
\begin{align*}
	N_{A,X,V,\delta}= \{u \in C^{4,\alpha}(\ol{\Omega}) :  \left(\Delta^2+ A \Delta + X\cdot \nabla +V\right)u=0,\, \mbox{and}\, \norm{u}_{C^{4,\alpha}(\ol{\Omega}} \le \delta\}.
\end{align*}


\begin{lemma}\label{lem_solution_map_for_both}
	Let $Q_i\in C^k(\Rb,C^{4,\alpha}(\ol{\Omega}), C^{4,\alpha}(\ol{\Omega};\Rnn),C^{4,\alpha}(\ol{\Omega}))$ for $i=1,2$. Let $w_i\in C^{4,\alpha}(\ol{\Om})$ be a solution of $ \Delta^2 w_i+ Q_i(x, w_i, \nabla w_i, \Delta w_i)=0$ with the same Cauchy data, for $i=1,2$. Let us denote, for $i=1,2$,
	\begin{align}
		\begin{cases}
			A_i(x) = \p_qQ(x,w_i(x),\nabla w_i(x),\Delta w_i(x)),\\ X_i(x) = \nabla_pQ(x,w_i(x),\nabla w_i(x),\Delta w_i(x)), \\ V_i(x) = \p_u Q(x,w_i(x),\nabla w_i(x),\Delta w_i(x)).
		\end{cases}
	\end{align}
	Let $S_{Q_1}: N_{A_1,X_1,V_1,\delta_1} \rightarrow C^{4,\alpha}(\ol{\Om}) $ be the solution map from Lemma \ref{lemma_perturbed_sol} for some $\delta_1>0$. Suppose $u_{1,v}= S_{Q_1} (v)$, and $ C_{Q_1}^{w, \delta}  \subseteq C_{Q_2}^{0,C}$. Then there exists $\delta_2>0$ and a $C^k$ map 
	$$T_{Q_2}: N_{A_1,X_1,V_1,\delta_2}\rightarrow C^{4,\alpha}(\ol{\Om}) \quad\text{given by } \ T_{Q_2} (v)= u_{2,v},$$
	where $u_{1,v}$ and $u_{2,v}$ have the same Cauchy data, and $u_{2,v}$ solves
	$ \Delta^2u_{2,v}+ Q_2(x, u_{2,v}, \nabla u_{2,v}, \Delta u_{2,v})=0$. Moreover, when one has $A_1=A_2, X_1=X_2$, and $V_1=V_2$, then $ D T_{Q_2}(0)v=v$.
\end{lemma}
\begin{proof}
	We first use the assumption $C_{Q_1}^{w_1, \delta}\subseteq  C_{Q_2}^{0,C}$ to have, for any $v\in N_{A_1,X_1,V_1,\delta_1}$, a function $u_{2,v}$ with the same Cauchy data as $u_{1,v}$. Next we observe that, $u_{1,0}=w_1$, and $u_{2,0}$ and $w_2$ have the same Cauchy data and solve the equation $\Delta^2 u + Q(x,u, \nabla u, \Delta u)=0$ in $\Om$.
	This together with Lemma \ref{lemma_cauchy_data_est} implies that $u_{2,0}=w_2$. Moreover, by \eqref{lemma_estimate_r} one has 
	\begin{align*}
		\norm{u_{1,v}-w_1}_{C^{4,\alpha}(\ol{\Om})}\le C   \norm{v}_{C^{4,\alpha}(\ol{\Om})}.
	\end{align*}
	The combination of this, Lemma \ref{lemma_cauchy_data_est}, $ (u_{2,v}-w_2)|_{\p\Om}=(u_{1,v}-w_1)|_{\p\Om}$,\, $ \Delta(u_{2,v}-w_2)|_{\p\Om}=\Delta (u_{1,v}-w_1)|_{\p\Om}$ $ \p_{\nu}(u_{2,v}-w_2)|_{\p\Om}=\p_{\nu}(u_{1,v}-w_1)|_{\p\Om}$,\, $ \p_{\nu}\Delta (u_{2,v}-w_2)|_{\p\Om}=\p_{\nu} \Delta(u_{1,v}-w_1)|_{\p\Om}$ and  $\norm{u_{2,v}}_{C^{4,\alpha}(\ol{\Om})}\le C$ entail 
	\begin{align*}
		\norm{u_{2,v}-w_2}_{C^{4,\alpha}(\ol{\Om})}\le C \norm{v}_{C^{4,\alpha}(\ol{\Om})}.
	\end{align*}
	Let us denote $r_v= u_{1,v}-u_{2,v}$, then $r_v$ satisfies
	\begin{equation}
	\begin{aligned}\label{eq_4.9}
		&(\Delta^2 + A_2\Delta + X_2\cdot \nabla +V_2)r_v \\&=  (A_2\Delta + X_2\cdot \nabla +V_2)r_v+ Q_2(x,u_{2,v},\nabla u_{2,v},\Delta u_{2,v})-Q_1(x,u_{1,v},\nabla u_{1,v},\Delta u_{1,v})\\&=
		(A_2\Delta + X_2\cdot \nabla +V_2)r_v-Q_1(x,u_{1,v},\nabla u_{1,v},\Delta u_{1,v})\\  &\quad+ Q_2(x,u_{1,v}-r_v,\nabla (u_{1,v}-r_v),\Delta (u_{1,v}-r_v)).
	\end{aligned}
	\end{equation}
	Let $G$ be the bounded inverse of $ (\Delta^2 + A_2\Delta + X_2\cdot \nabla +V_2): Y_2\rightarrow Z $ from Lemma \ref{lem_bdd_inv}. Then $r_v$  solves the fixed point equation:
	\begin{align*}
		r_v= G((A_2\Delta + X_2\cdot \nabla +V_2)r_v+ Q_2(x,u_{1,v}-r_v,\nabla (u_{1,v}-r_v),\Delta (u_{1,v}-r_v))-Q_1(x,u_{1,v},\nabla u_{1,v},\Delta u_{1,v})).
	\end{align*}
	
	Our next goal is to show the smooth dependence of $r_v$ on $v$. To this end we recall the projection  $P : C^{4,\alpha} (\ol{\Om})\rightarrow C^{4,\alpha} (\ol{\Om})\cap Z $ from Lemma \ref{lemma_projection}, because a  general function of the form $(A_2\Delta + X_2\cdot \nabla +V_2)r_v+ Q_2(x,u_{1,v}-r_v,\nabla (u_{1,v}-r_v),\Delta (u_{1,v}-r_v))-Q_1(x,u_{1,v},\nabla u_{1,v},\Delta u_{1,v})$ may not belong to the set $C^{4,\alpha} (\ol{\Om})\cap Z$. Now define the function 
	\begin{align*}
		F: N_{A_1,X_1,V_1,\delta_1}\times C^{4,\alpha} (\ol{\Om}) \rightarrow C^{4,\alpha} (\ol{\Om}) \quad \mbox{by}\quad F(v,r)= r- G P(f(v,r)),
	\end{align*}
	where $f(v,r)=(A_2\Delta + X_2\cdot \nabla +V_2)r+ Q_2(x,u_{1,v}-r,\nabla (u_{1,v}-r),\Delta (u_{1,v}-r))-Q_1(x,u_{1,v},\nabla u_{1,v},\Delta u_{1,v}).$ Next, we see that 
	\begin{equation}\label{eq_4.10}
		F(0,w_1-w_2)= w_1-w_2- GP(f(0,w_1-w_2)), \quad D_rF(0, w_1-w_2)h = h- D_r GP(f(0,w_1-w_2))h.
	\end{equation}
	The combination of \eqref{eq_4.9} and $u_{1,0}=w_1$ implies 
	\begin{align*}
		f(0,w_1-w_2)&= (A_2\Delta + X_2\cdot \nabla +V_2)(w_1-w_2)+ Q_2(x,w_2,\nabla w_2,\Delta w_2)-Q_1(x,w_1,\nabla u_{1,v},\Delta u_{1,v})\\
		&=(A_2\Delta + X_2\cdot \nabla +V_2)(w_1-w_2)+ \Delta^2(w_1-w_2)\\
		&= (\Delta^2+A_2\Delta + X_2\cdot \nabla +V_2)(w_1-w_2).
	\end{align*}
	In order to apply the projection $P$ to this we need $(\Delta^2+A_2\Delta + X_2\cdot \nabla +V_2)(w_1-w_2)\in C^{4,\alpha}(\ol{\Om})$. But by elliptic regularity this is true since the coefficients are $C^{4,\alpha}$ and $w_1,$ $w_2$ have the same Cauchy data and thus $w_1-w_2\in C^{8,\alpha}(\ol{\Om})$.
	This further entails 
	\begin{align*}
		F(0,w_1-w_2)&= w_1-w_2- GP((\Delta^2+A_2\Delta + X_2\cdot \nabla +V_2)(w_1-w_2))\\
		&= w_1-w_2- G((\Delta^2+A_2\Delta + X_2\cdot \nabla +V_2)(w_1-w_2))=0.
	\end{align*}
	We now focus on  computing the second term in \eqref{eq_4.10}. To this end,  we compute $D_r GP(f(0,w_1-w_2))h $. 
	\begin{align*}
		D_r& f(v,w_1-w_2))|_{v=0}h \\&=(A_2+X_2\cdot \nabla +V_2)h\\&\qquad- \p_{u} Q_2(x,u_{1,0}-(w_1-w_2),\nabla (u_{1,0}-(w_1-w_2)),\Delta (u_{1,0}-(w_1-w_2)) h\\&\qquad- \nabla_p Q_2(x,u_{1,0}-(w_1-w_2),\nabla (u_{1,0}-(w_1-w_2)),\Delta (u_{1,0}-(w_1-w_2))  \cdot \nabla h \\&\qquad-\p_{q} Q_2(x,u_{1,0}-(w_1-w_2),\nabla (u_{1,0}-(w_1-w_2)),\Delta (u_{1,0}-(w_1-w_2))\, \Delta h\\
		&= (A_2+X_2\cdot \nabla +V_2)h- \p_{u} Q_2(x,w_2,\nabla w_2,\Delta w_2) h\\&\qquad- \nabla_p Q_2(x,w_2,\nabla w_2,\Delta w_2)  \cdot \nabla h -\p_{q} Q_2(x,w_2,\nabla w_2,\Delta w_2)\, \Delta h =0.
	\end{align*}
	This implies $ D_rF(0,w_1-w_2)h =h.$
	It now follows from the implicit function theorem  that there exist $\delta_1$ and $\delta_2$ with $0 < \delta_2\le \delta_1$ and a  $C^k$ map $R: N_{A_1,X_1,V_1,\delta_2} \rightarrow C^{4,\alpha}(\ol{\Om})$ such that $R(v)= \tilde{r}$ is the unique solution of 
	\begin{align*}
		F(v, \tilde{r})=0\implies \tilde{r}= GP(f(v, \tilde{r})), \quad \mbox{for $\tilde{r}$ close to $w_1-w_2$}.   
	\end{align*}
	Next choosing $v\in N_{A_1,X_1,V_1,\delta_2}$ in $u_{1,v}= S_{Q_1}(v)$, we find that $r_v$ is in the range of $R$ and it satisfies the preceding equation. Moreover using the fact that $\norm{u_{i,v}-w_i}_{C^{4,\alpha}(\ol{\Om})}\le C \norm{v}_{C^{4,\alpha}(\ol{\Om})}$ for $i\in \{1,2\}$, we conclude that 
	\begin{align*}
		\norm{r_v-(w_1-w_2)}_{C^{4,\alpha}(\ol{\Om})}\le C \norm{v}_{C^{4,\alpha}(\ol{\Om})}.
	\end{align*}
	By the uniqueness of $\tilde{r}=R(v) $ near $w_1-w_2$ we have that $r_v=R(v)$, for $v\in N_{A_1,X_1,V_1, \delta_2}$ and consequently, the map $v \mapsto r_v$ is $C^k$. Now we are ready to define the map $T_{Q_2}$. The map 
	\begin{align*}
		T_{Q_2}\colon N_{A_1,X_1,V_1,\delta_2}\rightarrow C^{4,\alpha}(\ol{\Om}) \quad \mbox{ is defined as}\quad T_{Q_2}(v):= S_{Q_1}(v)- R(v).
	\end{align*}
	We now show that $DT_{Q_2}(0)h=h$, provided $A_1=A_2, X_1=X_2$, and $V_1=V_2$. To achieve this, we argue again using the implicit function theorem. We now differentiate $F(v,R(v))=0$ with respect to $v$ and evaluate at $v=0$ to conclude
	\begin{align*}
		DR(0)h = -[D_rF(0,R(0))]^{-1}D_vF(0,R(0))h.
	\end{align*}
	Since $D_rF(0,R(0))h =h$ and $D_vF(0,R(0))h=0$, this implies
	$ DR(0)h =0$ and consequently, we obtain $DT_{Q_2}(0)h=DS_{Q_1}(0)h-DR(0)h= h $.
\end{proof}

\section{First linearization}\label{sec:first_lin}
We start with recalling the following sets
\begin{align*}
	\V_{A,X,V}&= \{ u\in {C^{4,\alpha}(\ol{\Om})}: (\Delta^2+ A \Delta + X\cdot \nabla +V) u=0\},\\
	N_{A,X,V,\delta}&= \{u \in C^{4,\alpha}(\ol{\Omega}) :  \left(\Delta^2+ A \Delta + X\cdot \nabla +V\right)u=0,\, \mbox{and}\, \norm{u}_{C^{4,\alpha}(\ol{\Omega}} \le \delta\}.
\end{align*}
Let $u_{j,v}$ be the solutions of $ \Delta^2 u_{j,v} + Q_j(x,u_{j,v}(x),\nabla u_{j,v}(x),\Delta u_{j,v}(x))=0$ for any $v\in N_{A_1,X_1,V_1, \delta}$ with small $\delta >0$ given by Lemma \ref{lemma_fixed_point}, and Lemma \ref{lem_solution_map_for_both}, $j=1,2$.
\begin{lemma}\label{lm:determining_coeffi}
	Suppose $C_{Q_1}^{w, \delta}  \subseteq C_{Q_2}^{0,C}$. Then there is a $\delta_1>0$ such that for any $v\in N_{A_1,X_1,V_1, \delta_1} $ one has 
	\begin{align}
		A_1(x)=A_2(x), \,\, X_1(x)=X_2(x), \,\, V_1(x)= V_2(x)\quad \mbox{for $x\in \ol{\Omega}$}.
	\end{align}
\end{lemma}
\begin{proof}
	For $v \in N_{A_1,X_1,V_1, \delta_1}$, we write $v_t=v+th$ for small $t$, where $h$  solves the linearized equation 
	\begin{align*}
		( \Delta^2 + A_1\Delta + X_1 \cdot \nabla +V_1)h =0 \quad \mbox{in}\quad \Om.
	\end{align*}
	Denote $u_{1,v_t}=S_{Q_1,w}(v_t)$ and $u_{2,v_t}=T_{Q_2,w}(v_t)$, where $u_{j,v_t}$ solves $ \Delta^2 u_{j,v_t}+ Q_j(x,u_{j,v_t}, \nabla u_{j,v_t}, \Delta u_{j,v_t})=0$ for $j=1,2$. Since $u_{j,v_t}$ are $C^2$ in $t$ and have the same Cauchy data, differentiating $ \Delta^2 u_{j,v_t}+ Q_j(x,u_{j,v_t}, \nabla u_{j,v_t}, \Delta u_{j,v_t})=0$ with respect to $t$ and then inserting $t=0$ we obtain
	\begin{align*}
		0=   &\Delta^2\dot{u}_{j}+ \p_q Q(x,u_{j,v}, \nabla u_{j,v}, \Delta u_{j,v}) \Delta \dot{u}_{j} \\&+ \nabla_p  Q_j(x,u_{j,v}, \nabla u_{j,v}, \Delta u_{j,v}) \cdot \nabla \dot{u}_{j} + \p_{u} Q(x,u_{j,v}, \nabla u_{j,v}, \Delta u_{j,v}) \, \dot{u}_{j}
	\end{align*}
	where $\dot{u}_{j}= \p_t u_{j, v_t}|_{t=0}$. Now taking the difference between the above equation for $j=1$ and $j=2$, we derive
	\begin{equation}\label{eq_5.2}
		\begin{aligned}
			0=& P_2 (\dot{u}_{1}-\dot{u}_2) + \bigg(\p_q Q_1(x,u_{1,v}, \nabla u_{1,v}, \Delta u_{1,v})- \p_q Q_2(x,u_{2,v}, \nabla u_{2,v}, \Delta u_{2,v} )\bigg) \Delta \dot{u}_{1} \\& \quad+ \bigg(\nabla_p Q_1(x,u_{1,v}, \nabla u_{1,v}, \Delta u_{1,v})- \nabla_p Q_2(x,u_{2,v}, \nabla u_{2,v}, \Delta u_{2,v})\bigg) \cdot \nabla \dot{u}_{1} \\&\quad+\bigg(\p_u Q_1(x,u_{1,v}, \nabla u_{1,v}, \Delta u_{1,v})- \p_u Q_2(x,u_{2,v}, \nabla u_{2,v}, \Delta u_{2,v}) \bigg)\, \dot{u}_{1},
		\end{aligned} 
	\end{equation}
	where \[P_2u = \Delta^2u + \p_q Q_2(x,u_{2,v}, \nabla u_{2,v}, \Delta u_{2,v}) \Delta u 
	+ \nabla_p  Q_2(x,u_{2,v}, \nabla u_{2,v}, \Delta u_{2,v}) \cdot \nabla u + \p_u Q_2(x,u_{2,v}, \nabla u_{2,v}, \Delta u_{2,v} )u.\]
	Suppose $\tilde{v}_2$ solves $P^*_2(\tilde{v}_2)=0$, where $P^*_2$ is the $L^2$ adjoint of $P_2$ and the form of $P_2^*$  will the same  as $P_2$ with different coefficients.
	Now multiplying \eqref{eq_5.2} by $\tilde{v}_2$, and then utilizing integration by parts we conclude
	\begin{align*}
		0=&\int_{\Om}  \bigg(\p_q Q_1(x,u_{1,v}, \nabla u_{1,v}, \Delta u_{1,v})- \p_q Q_2(x,u_{2,v}, \nabla u_{2,v}, \Delta u_{2,v} ) \bigg)\Delta \dot{u}_{1}\, \tilde{v}_2\\& \quad+ \bigg(\nabla_p Q_1(x,u_{1,v}, \nabla u_{1,v}, \Delta u_{1,v})- \nabla_p Q_2(x,u_{2,v}, \nabla u_{2,v}, \Delta u_{2,v})\bigg) \cdot \nabla \dot{u}_{1}\, \tilde{v}_2 \\&\quad+\bigg(\p_u Q_1(x,u_{1,v}, \nabla u_{1,v}, \Delta u_{1,v})- \p_u Q_2(x,u_{2,v}, \nabla u_{2,v}, \Delta u_{2,v}) \bigg)\, \dot{u}_{1}\, \tilde{v}_2.  
	\end{align*}
	It remains to study $ \dot{u}_1= DS_{Q_1,w}(v)h$. By Lemma \ref{lem_differomorphism} when $ v \in N_{A_1,X_1,V_1, \delta_1}$, any sufficiently small solution $\tilde{v}_1$ of   \[(\Delta^2 + \p_q Q(x,u_{1,v}, \nabla u_{1,v}, \Delta u_{1,v} \Delta+ \nabla_p  Q(x,u_{1,v}, \nabla u_{1,v}, \Delta u_{1,v} \cdot \nabla + \p_u Q(x,u_{1,v}, \nabla u_{1,v}, \Delta u_{1,v} )(\cdot)=0\]  can be written as $DS_{Q_1,w}(v)h$ for suitable $h$. This further entails
	\begin{equation*}
		\begin{aligned}
			0=&\int_{\Om}  \bigg(\p_q Q_1(x,u_{1,v}, \nabla u_{1,v}, \Delta u_{1,v})- \p_q Q_2(x,u_{2,v}, \nabla u_{2,v}, \Delta u_{2,v} ) \bigg)\Delta \tilde{v}_1\, \tilde{v}_2\\& \quad+ \bigg(\nabla_p Q_1(x,u_{1,v}, \nabla u_{1,v}, \Delta u_{1,v})- \nabla_p Q_2(x,u_{2,v}, \nabla u_{2,v}, \Delta u_{2,v})\bigg) \cdot \nabla \tilde{v}_1\, \tilde{v}_2 \\&\quad+\bigg(\p_u Q_1(x,u_{1,v}, \nabla u_{1,v}, \Delta u_{1,v})- \p_u Q_2(x,u_{2,v}, \nabla u_{2,v}, \Delta u_{2,v}) \bigg)\, \tilde{v}_1\, \tilde{v}_2.
		\end{aligned}
	\end{equation*}
	Now using  \cite[Theorem 1.1]{BKS_mrt_polyharmonic} in  dimensions $n \ge 3$, and using \cite{BansalKrishnanPattar} in two dimensions, we can conclude that 
	\begin{align}\label{eq_5.3}
		\begin{cases}
			\p_q Q(x,u_{1,v}, \nabla u_{1,v}, \Delta u_{1,v})= \p_q Q(x,u_{2,v}, \nabla u_{2,v}, \Delta u_{2,v}),\\
			\nabla_p Q(x,u_{1,v}, \nabla u_{1,v}, \Delta u_{1,v})=  \nabla_p Q(x,u_{2,v}, \nabla u_{2,v}, \Delta u_{2,v}),\\
			\p_u Q(x,u_{1,v}, \nabla u_{1,v}, \Delta u_{1,v})=  \p_u Q(x,u_{2,v}, \nabla u_{2,v}, \Delta u_{2,v}).
		\end{cases}
	\end{align}
	This completes the proof.
\end{proof}

\begin{lemma}\label{lem_v_independency}
	In the context of Lemma \ref{lm:determining_coeffi}, we see that the function $\phi_v= u_{2,v}-u_{1,v}$ is independent of $v \in N_{A_1,X_1,V_1, \delta_1}$.
\end{lemma}
\begin{proof}
	Define the function $f(t,x) = \phi_{tv}(x)$ for $x\in \ol{\Om}$. Then $f(t,x)$ is smooth in $t$ and has zero Cauchy data on $\p\Om$. Moreover it satisfies
	\begin{align*}
		\Delta^2 f(t,x) = -Q_2(x, u_{2,tv}, \nabla u_{2,tv} , \Delta u_{2,tv})+ Q_1(x, u_{1,tv}, \nabla u_{1,tv} , \Delta u_{1,tv}).
	\end{align*}
	Differentiating the above equation with respect to $t$, and then using  \eqref{eq_5.3} we obtain
	\begin{align*}
		\Delta^2 g(t,x)&=  \p_q Q_2(x, u_{2,tv}, \nabla u_{2,tv} , \Delta u_{2,tv}) \p_t \Delta (u_{1,tv}- u_{2,tv}) \\&  \quad + \nabla_p Q_2(x, u_{2,tv}, \nabla u_{2,tv} , \Delta u_{2,tv}) \cdot \p_t\nabla (u_{1,tv}- u_{2,tv}) \\& \quad+ \p_u Q_2(x, u_{2,tv}, \nabla u_{2,tv} , \Delta u_{2,tv})\, \p_t(u_{1,tv}- u_{2,tv}),
	\end{align*}
	where $ g(t,x) = \p_t f(t,x) = \p_t ( u_{2,tv}- u_{1,tv})$. This entails
	\begin{align*}
		0&=   \Delta^2 g(t,x)+  \p_q Q_2(x, u_{2,tv}, \nabla u_{2,tv} , \Delta u_{2,tv})  \Delta g(t,x)   \\& \quad+ \nabla_p Q_2(x, u_{2,tv}, \nabla u_{2,tv} , \Delta u_{2,tv}) \cdot \nabla g(t,x)+ \p_u Q_2(x, u_{2,tv}, \nabla u_{2,tv} , \Delta u_{2,tv})\, g(t,x), 
	\end{align*}
	Since $g(t,x)$ has zero Cauchy data, this implies $g(t,x)=0$ and consequently $f(t,x)$ is independent of $t$. In particular, $\phi_v= \phi_0$.
\end{proof}
\section{proof of main results}\label{sec:proof_of_main}
In this section we present the proof of Theorem \ref{th_main_upto_gauge}. Then the proof of Theorem \ref{Th:main} follows from the proof of  Theorem \ref{th_main_upto_gauge} 
\begin{proof}[Proof of Theorem \ref{th_main_upto_gauge}]
	Suppose  $w_1$ solves $ \Delta^2 w_1+ Q_1(x, w_1, \nabla w_1, \Delta w_1)=0 $ and assume that $C_{Q_1}^{w_1, \delta}\subseteq C_{Q_1}^{0,C}$. Using Lemma \ref{lem_v_independency}, we have
	\begin{align*}
		\Delta^2\varphi&=\Delta^2 u_{2,v} - \Delta^2 u_{1,v}\\
		&= Q_1(x,u_{1,v},\nabla u_{1,v},\Delta u_{1,v}) - Q_2(x,u_{2,v},\nabla u_{2,v},\Delta u_{2,v})\\
		&= Q_1(x,u_{1,v},\nabla u_{1,v},\Delta u_{1,v}) - Q_2(x,u_{1,v}+\varphi,\nabla (u_{1,v}+\varphi),\Delta (u_{1,v}+\varphi))
	\end{align*}
	which implies
	\begin{equation*}
		Q_1(x,u_{1,v},\nabla u_{1,v},\Delta u_{1,v}) = \Delta^2\varphi + Q_2(x,u_{1,v}+\varphi,\nabla (u_{1,v}+\varphi),\Delta (u_{1,v}+\varphi))=T_{\phi}Q_2(x,u_{1,v}, \nabla u_{1,v}, \Delta u_{1,v}) .
	\end{equation*}
	To conclude the proof, we need to show that there exists an $\e>0$ such that for any $\tilde x\in\ol{\Om}$ and $\lambda\in [-\e,\e]$ we can find a small solution of the linearized equation so that
    \begin{equation}
        \begin{aligned}\label{eq_u_plus_lambda}
		u_{1,v}(\tilde x)=w_1(\tilde x) + \lambda, \quad
		\nabla u_{1,v}(\tilde x)=\nabla w_1(\tilde x) + \ol{\lambda},\quad
		\Delta u_{1,v}(\tilde x)=\Delta w_1(\tilde x) + \lambda.
	\end{aligned}
    \end{equation}
	We begin by fixing $x_1\in\ol{\Om}$, and then use Runge approximation from Lemma \ref{lemma_Runge_approximation} to generate a solution $v_{x_1}$ of $\Delta^2 v + A_1\Delta v + X_1\cdot \nabla v + V_1 v=0$ such that 
	\begin{equation*}
		v_{x_1}(x_1)=4,\quad \Delta v_{x_1}(x_1)=4,\quad \p_{x_k}v_{x_1}(x_1)=4\quad\text{for all } 1\leq k\leq n.
	\end{equation*}
	Let $U_{x_1}$ be a neighborhood of $x_1$ in which $v_{x_1}(x), \Delta v_{x_1}(x), \p_{x_k}v_{x_1}(x)\geq 2$, $1\leq k\leq n$, for all $x\in \ol{U}_{x_1}\cap\ol{\Om}$. Notice that by Lemma \ref{lemma_fixed_point} we can write $u_{1,tv_{x_1}} = w_1+tv_{x_1}+S_{Q_1,w_1}(tv_{x_1})$, where $\norm{S_{Q_1,w_1}(tv_{x_1})}_{C^{4,\alpha}(\ol{\Om})}\le C t^2 \norm{v_{x_1}}^2_{C^{4,\alpha}(\ol{\Om})}$. Then
	\begin{align*}
		\abs{u_{1,tv_{x_1}}(x)-w_1(x)} &\geq 2\abs{t} - Ct^2 \norm{v_{x_1}}^2_{C^{4,\alpha}(\ol{\Om})}\\
		\abs{\Delta u_{1,tv_{x_1}}(x)-\Delta w_1(x)} &\geq 2\abs{t} - Ct^2 \norm{v_{x_1}}^2_{C^{4,\alpha}(\ol{\Om})}\\
		\abs{\p_{x_k}u_{1,tv_{x_1}}(x)-\p_{x_k}w_1(x)} &\geq 2\abs{t} - Ct^2 \norm{v_{x_1}}^2_{C^{4,\alpha}(\ol{\Om})}
	\end{align*}
	for $x\in \ol{U}_{x_1}\cap\ol{\Om}$ and for $1\leq k \leq n$. For $\abs{t}\leq \e_{x_1}:=\frac{1}{Ct^2\norm{v_{x_1}}^2_{C^{4,\alpha}(\ol{\Om})}}$ we get
	\begin{align*}
		\abs{u_{1,tv_{x_1}}(x)-w_1(x)} &\geq \abs{t}\\
		\abs{\Delta u_{1,tv_{x_1}}(x)-\Delta w_1(x)} &\geq \abs{t}\\
		\abs{\p_{x_k}u_{1,tv_{x_1}}(x)-\p_{x_k}w_1(x)} &\geq \abs{t}.
	\end{align*}
	We choose a finite cover ${U_{x_1},\ldots,U_{x_m}}$ of $\ol{\Om}$ and set $\delta_0$ such that $\norm{tv_{x_j}}^2_{C^{4,\alpha}(\ol{\Om})}\leq \delta$ whenever $\abs{t}\leq\delta_0$ and $1\leq j\leq m$. Then define
	\begin{equation*}
		\e:=\min\{\delta_0,\e_{x_1},\ldots,\e_{x_m}\}.
	\end{equation*}
	Next, fix $\tilde x\in\ol{\Om}$, $\lambda\in[-\e,\e]$ and choose $j$, so that $\tilde x\in U_{x_j}$. Then define $\rho, \rho_{\Delta}, \rho_k\colon [-\e,\e]\to \Rb,$
	\begin{align*}
		\rho(t)=u_{1,tv_{x_j}}(\tilde x) - w_1(\tilde x),\,\,
		\rho_{\Delta}(t)=\Delta u_{1,tv_{x_j}}(\tilde x) - \Delta w_1(\tilde x),\,\,
		\rho_k(t)=\p_{x_k}u_{1,tv_{x_j}}(\tilde x) - \p_{x_k}w_1(\tilde x),
	\end{align*}
	where $1\leq k \leq n$. Then these are all continuous with respect to $t$ and satisfy
    \[\rho(-\e), \rho_{\Delta}(-\e), \rho_k(-\e) \leq -\e \quad\mbox{and} \quad \rho(\e), \rho_{\Delta}(\e), \rho_k(\e) \geq \e \quad\mbox{for all $1\leq k \leq n$}.\] Hence by continuity there exists $\tilde t, \tilde{t}_{\Delta}, \tilde{t}_k \in [-\e,\e]$ such that
	\begin{equation*}
		\rho(\tilde t)=\rho_{\Delta}(\tilde t_{\Delta})=\rho_k(\tilde t_k)=\lambda.
	\end{equation*}
	Thus \eqref{eq_u_plus_lambda} is satisfied for some $v$, and the proof is finished.
\end{proof}

\section{Runge approximation}\label{sec:runge}
This section is devoted to proving the Runge approximation result for a fourth order biharmonic equation. A linear differential operator $P$ satisfies the \emph{Runge Property} if the adjoint operator $P^t$ satisfies  the \emph{Unique Continuation Property}, see \cite{Lax1956}. The notion of classical Runge approximation property dates back to the work of \cite{Lax1956,Malgrange1956}.  We also refer to some recent works \cite{RS_Quantative_Runge_I,RS_Quantative_Runge_II,LLS_Math_Annalen,JNS2023} where Runge approximation property is used to solve certain inverse problem. In this paper we are interested to a  quantitative version of Runge property and it is somehow related to the density of solutions of elliptic partial differential equations in large domain into a smaller domain with respect to the $L^2$-norm in the smaller domain. This can be achieved using Hahn-Banach theorem. A general set-up for such approximation results can be found in \cite{Browder_AJM}. 

To this end,  we recall that  $ Lu=(\Delta^2+ A\Delta + X \cdot \nabla + V) u.$
\begin{lemma}
	Let $\Om\subset \Rnn$ be a bounded domain with smooth boundary, and let $ A\in C^{2,\alpha}(\ol{\Om};\Rb^{n^2}), X\in C^{1,\alpha}(\ol{\Om};\Rnn)$ and $V\in C^{\alpha}(\ol{\Om})$. Assume that $0$ is not a Dirichlet eigenvalue of $L$ in $\Om$. Then for any $x_0\in \ol{\Om}$, there is a solution $u\in C^{4,\alpha}(\ol{\Om})$ of $L u=0$ in $\Om$ such that $u(x_0)\neq 0, \nabla u(x_0)\neq 0, \nabla^2 u(x_0)\neq 0$.
\end{lemma}
\begin{proof}
	Since $ A\in C^{2,\alpha}(\ol{\Om};\Rb^{n^2}), X\in C^{1,\alpha}(\ol{\Om};\Rnn)$ and $V\in C^{\alpha}(\ol{\Om})$, we may choose a ball  $B$ such that $\ol{\Om}\subset B$ and extend  $A, X, V$ in such a way that  $A\in C_c^{2,\alpha}(B), X\in C_c^{1,\alpha}(B;\Rnn)$ and $V\in C_c^{\alpha}(B)$. We can also choose $B$ in such a way that $0$ is not an eigenvalue of the operator $Lu=(\Delta^2+ A\Delta + X \cdot \nabla + V) u$. This can be achieved considering the symmetric part of $L$ and using the continuous dependency of eigenvalues with respect to domains; see for instance  \cite{Babuska_1965,stefanov1990,Math222B_Spring2023,Bent_JFA}.
	
	Next by utilizing \cite[Theorem 1 in Section 5.4 of Part 2]{bjs} there is a small neighborhood $\Om_{x_0}$ of $x_0$ and a function $u\in C^{4,\alpha}(\ol{\Om}_{x_0})$ solving $Lu=(\Delta^2+ A\Delta + X \cdot \nabla + V) u=0$ in $\Om_{x_0}$ such that $u(x_0)=\p_{x_k}u(x_0)=\Delta u(x_0)=1$, $1\leq k\leq n$. By Lemma \ref{lemma_Runge_approximation} there exists a $\tilde u\in C^{4,\alpha}(\ol{\Om})$ solving $(\Delta^2+ A\Delta + X \cdot \nabla + V) \tilde u=0$ in $\Om$ with $\tilde u(x_0), \p_{x_k}\tilde u(x_0), \Delta \tilde u(x_0)$ being as close to $u(x_0)=\p_{x_k}u(x_0)=\Delta u(x_0)=1$ as possible. 
\end{proof}

Let $L, A, X$ and $V$ be as before and assume that $0$ is not a Dirichlet eigenvalue of $L$ in $\Om$. For a bounded linear functional $\mu$ on $C^2(\ol{\Om})$ we say that $u\in L^1(\Om)$ is a \emph{very weak solution} of 
\begin{equation}\label{BVP_very_weak}
	\begin{cases}
		L^*u = \mu & \quad\text{in }\Om \\
		u=\Delta u= 0&\quad\text{on }\p\Om
	\end{cases}
\end{equation}
if $	\int_{\Om}uL\varphi\,dx = \mu(\varphi)$, 
for any $\varphi\in C^4(\ol{\Om})\cap H^4_N(\Om)$ where $H^4_N(\Om):=\{v\in H^4(\Om) :  \varphi|_{\p\Om}=\Delta\varphi|_{\p\Om}=0\}$.
We next show the existence of a very weak solution to \eqref{BVP_very_weak}.

\begin{proposition}\label{Prop_7.2}
	For any bounded linear functional $\mu$ on $C^2(\ol{\Om})$ there exists a very weak solution $u\in W^{2,p}(\Om)$, $p<\frac{n+2}{n}$, to \eqref{BVP_very_weak}.
\end{proposition}
\begin{proof}
	Since $0$ is not an eigenvalue of $L$ (see \cite[Equation 2.17]{Gazzola_book}), by \cite{Gazzola_book}*{Theorem 2.20} the following boundary value problem has unique solution $u\in W^{4,q}(\Om)$ for every $f\in L^q(\Om)$
	\begin{equation}\label{L_adjoint}
		\begin{cases}
			Lu = f & \quad\text{in }\Om \\
			u=\Delta u= 0&\quad\text{on }\p\Om.
		\end{cases}
	\end{equation}
    Now take $q$ such that $2q>n+2$, the condition $2q>n+2$ implies that $ W^{4,q}(\Omega)\subset W^{2,q}(\Omega)\subset C^2(\ol{\Omega})$
    Let us denote by $X:=\mathrm{Range}(L)\subset L^p(\Omega)$. Then for every $\varphi \in X$ we have $L\psi=\varphi$ for some $\psi\in W^{4,q}_N(\Om)\cap C^2(\ol{\Omega})$.  For a fixed continuous linear functional $\mu$ on $C^2(\ol{\Om})$ define an operator
	\begin{equation*}
		B (L\psi):=\int_{\Om}\psi\,d\mu,
	\end{equation*}
	and for this operator we have
	\begin{equation*}
		\abs{B (L\psi)}\leq C\norm{\mu}_{*}\norm{\psi}\leq C\norm{\mu}_{*}\norm{ L\,\psi}.
	\end{equation*}
	Here $\norm{\mu}_{*} = \sup_{\norm{f}_{C^2(\ol{\Omega})}=1} \abs{\mu(f)} $.
	Hence $B$ is bounded on $X\subset L^q(\Omega) $, and by Hahn-Banach it is bounded on $L^q(\Om)$. Hence the Riesz representation gives the existence of $u\in L^p(\Om)$ with the property that $p^{-1}+q^{-1} =1$ and 
	\begin{equation*}
		\int_{\Om}u\varphi\,dx = B(\varphi)\quad\text{for all }\varphi\in L^q(\Om).
	\end{equation*}
	For any $\psi\in  C_c^{\infty}(\Omega)$ we have $	B(L\psi)= \int\psi\, \D \mu=	\int u\,L\psi\, = \int L^*u\, \psi.$
	This implies $L^*u=\mu $ in $\Omega$,  where $u\in L^p(\Omega)$ with  $p<\frac{n+2}{n}$, since $q>\frac{n+2}{2}$. 
	
	Moreover, for open set $\Omega'\subset \subset \Omega$ one could proceed in the same  way to show that $u\in W^{2,p}(\Omega')$. This can be done using the fact that $L$ is an elliptic differential operator of  order $4$ and therefore it has a parametrix which is a  pseudodifferential operator of order $-4$ ; see \cite[Corollary 4.3]{Treves_vol_1}. Then the combination of  these facts along with \cite[Theorem 2.1]{Treves_vol_1}  conclude the proof.
\end{proof}

We now prove the following Runge approximation result.  To this end,  Let us consider two open sets $\Om_1\subset \Om$  in $\Rnn$ and also consider the sets
\begin{align*}
	N(\Om_1)=\{u \in C^{4,\alpha}(\ol{\Om_1}): Lu=0\,\, \mbox{in $ \Om_1$} \},\,\quad\mbox{and}\quad
	N(\Om) =\{u \in C^{4,\alpha}(\ol{\Om}): L u=0\,\, \mbox{in $ \Om$} \}.
\end{align*}

\begin{lemma}\label{lemma_Runge_approximation}
	Let $\Om_1\subset \Om$ be two open subsets in $\Rnn$ such that $\Om\setminus \ol{\Om}_1$ is connected. Suppose $A\in C_c^{2,\alpha}(\Om), X\in C_c^{1,\alpha}(\Om;\Rnn)$ and $V\in C_c^{\alpha}(\Om)$ and $0$ is not a Dirichlet eigenvalue of $ (\Delta^2+ A\Delta + X \cdot \nabla + V)$ in $\Om$. Then  $N(\Om)$ is dense in $N(\Om_1)$ with $C^2(\ol{\Om}_1) $ norm. Moreover, given any  $u\in N(\Om_1)$ and any $\epsilon>0$, there exists $v \in N(\Om)$ such that  $\norm{u-v|_{\Om_1} }_{C^2(\ol{\Om}_1)}\le \epsilon$.
\end{lemma}
\newcommand{\omm}{\Omega \setminus \ol{\Omega}_1}
\begin{proof}
	We use the Hahn-Banach theorem to prove this result. Let $\mu$ be a linear functional  on $C^2(\ol{\Omega}_1)$ that satisfies $ \mu(v|_{\ol{\Omega}_1})=0$ for all $v\in N(\Omega)$. We now extend $\mu$, denoted by $\ol{\mu}$, as a continuous linear functional on $C^2(\ol{\Omega})$ in the following way. Define $\bar{\mu}: C^2(\ol{\Omega}) \rightarrow \Rb$ such that
	\begin{align}
		\ol{\mu}(u):= \mu (u|_{\ol{\Omega}_1}) \quad \mbox{for all $u\in C^2(\ol{\Omega})$}.
	\end{align}
	By Riesz representation  theorem, $\ol{\mu}$ is a Radon measure on $\ol{\Omega}$, and by the definition we have  $\operatorname{supp} \ol{\mu} \subset \ol{\Omega}_1 $.  Next we utilize Proposition \ref{Prop_7.2} to find a very weak solution $\psi\in L^p(\Omega)$  of 
	\begin{align*}
		L^*\psi&=\ol{\mu} \quad \mbox{in $\Omega$}\quad \mbox{and}, \quad 
		(\psi, \Delta \psi)=0 \quad \mbox{on $\p\Omega$}.
	\end{align*}
	Our next goal is to show that 	\begin{align}\label{eq_7.4}
		\psi=0 \quad \mbox{in $\Omega \setminus \ol{\Omega}_1$}.
	\end{align}
	\textbf{Step 1.} Assume that \eqref{eq_7.4}  holds. We then claim that $\mu|_{N(\Omega_1)} =0$.\\\smallskip
	Since $\operatorname{supp}\psi \subset  \omm$, there exists ${\psi_j}\in C_c^{\infty}(\Omega_1)$ with $\psi_j \rightarrow \psi$ in $L^{p}(\Omega)$.  Given any $u\in N(\Omega_1)$ we  can choose $\ol{u}\in C^{4, \alpha} (\ol{\Omega})$ such that $ \ol{u}= u$ on $\ol{\Omega}_1$. This implies
	\begin{align*}
		\mu(u) = \ol{\mu}  (\ol{u})= \int_{\Omega} \psi \,  L\ol{u}=  \lim_j \int_{\Omega} \psi_j \, L \ol{u}= \int_{\Omega_1} \psi_j \,  L u=0,
	\end{align*}
	and concludes that $\mu|_{N(\Omega_1)} =0$.\\
	\textbf{Step 2.} We now prove \eqref{eq_7.4}.\\\smallskip
	To prove this, we  use  the fact that support of $\ol{\mu}$ is contained in $\ol{\Omega}_1$, and combine this  with unique continuation principle.  Since $\operatorname{supp}{\ol\mu}\subset \ol{\Omega}_1$, this implies $\psi$ solves  
	$L^* \psi=0$ in $\Omega \setminus \ol{\Omega}_1$.
	
	Now we are going to use the fact that the coefficients  $A,X,V$ are compactly supported in $\Omega$. Let us assume that support of $A,X,V$ are contained in a compact set $K\subset \Omega$. Then we choose another open set $\Omega_2$ such that $K, \Omega_1\subset \Omega _2\subset \Omega$. Since $\Omega\setminus \Omega_2$ is connected and $\psi$ solves $  \Delta^2 \psi=0 \quad \mbox{in $\Omega\setminus \Omega _2$}$. This  implies that $\psi\in C^{4,\alpha}$ near $\p\Omega$. Infact, $\psi$ is $C^{4,\alpha}$ in $\omm$. We now let $v\in N(\Omega)$ and choose a cut-off function $\chi$ in such a way that $\chi =1$ near $\Omega_1$ and $\psi\in C^{4,\alpha}$  in  $\operatorname{supp}(1-\chi) \cap \ol{\Om}$.  Utilizing these facts we deduce
	\begin{align}\label{eq_7.5}
		0&= \int_{\Omega} \psi  Lv=
		\int_{\Omega} \psi  L (\chi\, v) +  \int_{\Omega} \psi  L (1-\chi) v
	\end{align}
	Next using the definition of very weak solution, the first term of above we can write as:
	\begin{align}\label{eq_7.6}
		\int_{\Omega} \psi  L (\chi\, v) = \int_{\Omega} \psi  (\Delta^2+A\Delta + X \cdot \nabla + V)(\chi\, v)= \ol{\mu}(\chi v)= \ol{\mu}(\chi v).
	\end{align} 
	Using integration by parts, $\operatorname{supp}(1-\chi)v \subset \Omega\setminus\ol{\Omega}_1$, and the fact that $\psi$ is regular in $\omm$, we can write $\int_{\Omega} \psi L(1-\chi) v=  \int_{\Om} L^* \psi\, (1-\chi) v+  \int_{\p\Omega} \p_{\nu} \Delta \psi\, v+ \int_{\p\Omega}\p_{\nu} \psi\, \Delta v $. We combine this with \eqref{eq_7.5} and \eqref{eq_7.6} to conclude
	\begin{align*}
		0=\ol{\mu}(\chi \,v)  +\int_{\p\Omega} \p_{\nu} \Delta \psi\, v+ \int_{\p\Omega} \Delta \psi\, \p_{\nu} v= \mu(v|_{\ol{\Omega}_1})+\int_{\p\Omega} \p_{\nu} \Delta \psi\, v+  \int_{\p\Omega}\p_{\nu} \psi\, \Delta v .
	\end{align*}
	Since  $v\in N(\Omega)$, then by the assumption we have that $\mu(v|_{\ol{\Omega}_1})=0$. Next choosing specific  boundary values of $v$ we conclude that  $ \p_{\nu}\Delta  \psi=\p_{\nu}\psi=0 $ on $\p\Omega$. Therefore we have that $\psi$ satisfies
	\begin{equation}
		\begin{aligned}
			L_{A,X,V}^* \psi=0 \quad \mbox{in $\omm$ with $(\psi,\p_{\nu} \psi, \p_{\nu}^2\psi, \p^3_{\nu}\psi)=0$ on $\p\Om$.}
		\end{aligned}
	\end{equation}
	This along with unique continuation principle gives $\psi=0$ in $\omm$; see for instance \cite{Vessella_KL_eq_ucp,Vessella_size_estimate_EP} and the references therein. This completes the proof.
\end{proof}
\appendix
\section{computational details}\label{sec:appendix}
\setstretch{1}

In this section, we provide the computational details necessary to complete the proof. Due to the length of the proof and to ensure a smooth reading experience, we have chosen to include these details in the appendix.
\begin{lemma}\label{Q_deriv_estimate}
	Let $Q\in C^k(\Rb,C^{4,\alpha}(\ol{\Omega}), C^{3,\alpha}(\ol{\Omega};\Rnn),C^{2,\alpha}(\ol{\Omega})),$ and $l\le k$. Suppose $M=\norm{f}_{C^2(\ol{\Om})}$. Then the following holds true.
	\begin{itemize}
		\item [1.] $ \norm{\p^l_uQ(x,f(x),\nabla f(x), \Delta f(x))}_{C^{\alpha}(\ol{\Omega})}\le \norm{Q}_{C^k([-M,M],[-M,M]^n,[-M,M], C^{\alpha}(\ol{\Om}))}$.
		\item [2.] $\norm{\p^l_{p_i}Q(x,f(x),\nabla f(x), \Delta f(x))}_{C^{\alpha}(\ol{\Omega})}\le \norm{Q}_{C^k([-M,M],[-M,M]^n,[-M,M], C^{\alpha}(\ol{\Om}))}$ for some $1\le i\le n$.
		\item [3.] $\norm{\p^l_qQ(x,f(x),\nabla f(x), \Delta f(x))}_{C^{\alpha}(\ol{\Omega})}\le \norm{Q}_{C^k([-M,M],[-M,M]^n,[-M,M], C^{\alpha}(\ol{\Om}))}$.
		\item[4.] $\norm{\p^m_{u}\p^j_{p_i}Q(x,f(x),\nabla f(x), \Delta f(x))}_{C^{\alpha}(\ol{\Omega})}\le \norm{Q}_{C^k([-M,M],[-M,M]^n,[-M,M], C^{\alpha}(\ol{\Om}))}$, whenever $m+j\le k$.
	\end{itemize}
	The last assertion holds true for other mixed derivatives also.
\end{lemma}
\begin{proof}
	We have 
	\begin{align*}
		&\norm{\p^l_uQ(x,f(x),\nabla f(x), \Delta f(x))}_{C^{\alpha}(\ol{\Omega})}\\ &=  \sup_{x\in \ol{\Omega}} \abs{\p^l_uQ(x,f(x),\nabla f(x), \Delta f(x))} \\
		&+  \sup_{x,y\in \ol{\Omega}, x\neq y}  \frac{\abs{\p^l_uQ(x,f(x),\nabla f(x), \Delta f(x))- \p^l_uQ(y,f(y),\nabla f(y), \Delta f(y))}} {\abs{x-y}^{\alpha}}. 
\end{align*}
Using $M=\norm{f}_{C^2(\ol{\Om})}$, we can obtain an upper bound  of  $\sup_{x\in \ol{\Omega}} \abs{\p^l_uQ(x,f(x),\nabla f(x), \Delta f(x))}  $ which is  $ \sup_{x\in \ol{\Omega}} \norm{\p^l_uQ(x,\nabla f(x), \Delta f(x))}_{C[-M,M]}$. Similarly, $ \sup_{x\in \ol{\Omega}} \norm{\p^l_uQ(x,\nabla f(x), \Delta f(x))}_{C[-M,M]}$ can be bounded by  $ \sup_{x\in \ol{\Omega}} \norm{\p^l_uQ(x)}_{C([-M,M],[-M,M]^n,[-M,M])}$. Moreover,  using $l\le k$, we obtain 
\begin{align}\label{eq_a_1}
	\sup_{x\in \ol{\Omega}} \abs{\p^l_uQ(x,f(x),\nabla f(x), \Delta f(x))} \le  \sup_{x\in \ol{\Omega}} \norm{Q(x)}_{C^k([-M,M],[-M,M]^n,[-M,M])}.
\end{align}
Similarly, one can bound $ \sup_{x,y\in \ol{\Omega}, x\neq y}  \frac{\abs{\p^l_uQ(x,f(x),\nabla f(x), \Delta f(x))- \p^l_uQ(y,f(y),\nabla f(y), \Delta f(y))}} {\abs{x-y}^{\alpha}}$ by 
\begin{align*}
	& \sup_{x,y\in \ol{\Omega}, x\neq y}  \frac{\abs{\p^l_uQ(x,f(x),\nabla f(x), \Delta f(x))- \p^l_uQ(y,f(y),\nabla f(y), \Delta f(y))}} {\abs{x-y}^{\alpha}} \\
	\le &  \sup_{x,y\in \ol{\Omega}, x\neq y}  \frac{\norm{Q(x) -Q(y)}_{C^k([-M,M],[-M,M]^n,[-M,M])}} {\abs{x-y}^{\alpha}}. 
\end{align*}
The combination preceding estimates and \eqref{eq_a_1} concludes the proof. The proof for estimating $C^{\alpha}{(\ol{\Omega})}$ norm of other mixed derivatives  works analogously.
\end{proof}
\setstretch{0.9}
\begin{lemma}\label{lem_R_v}
The function $R_v(r) =R(v+r)$ given by 
\begin{equation*}
	\begin{aligned}
		R(h) &:= \int_0^1 [\p_u Q(x,w+th,\nabla(w+th),\Delta(w+th)) - \p_u Q(x,w,\nabla w,\Delta w)]h \,\D t \\
		& \quad+ \int_0^1 [\nabla_p Q(x,w+th,\nabla(w+th),\Delta(w+th)) - \nabla_p Q(x,w,\nabla w,\Delta w)]\cdot \nabla h \,\D t\\
		&\quad+ \int_0^1 [\p_q Q(x,w+th,\nabla(w+th),\Delta(w+th)) - \p_q Q(x,w,\nabla w,\Delta w)] \Delta h \,\D t.
	\end{aligned}
\end{equation*} satisfies 
\begin{align*}
		& R_v(r_1) -R_v(r_2) \\
		&=\int_0^1 (u_1-u_2) \int_0^1 \frac{d}{\D s}\p_u Q(x,w+stu_1,\nabla(w+stu_1),\Delta(w+stu_1))\D s  \\& \qquad +\int_0^1 (u_1-u_2) \int_0^1 \frac{d}{\D s}\p_u Q(x,w+stu_2,\nabla(w+stu_2),\Delta(w+stu_2))\D s \D t\\
		&\qquad+  \int_0^1 \int_0^1  s t^2 u_1\,u_2\int_0^1 (u_1-u_2)\, \p^3_u Q(x,w+z(\tau),\nabla(w+z(\tau)),\Delta(w+z(\tau))) \D \tau \D s \, \D t\\&\qquad+ 
		\int_0^1 \int_0^1  s t^2 u_1\,u_2\int_0^1 \nabla (u_1-u_2) \cdot \nabla_p \p^2_u Q(x,w+z(\tau),\nabla(w+z(\tau)),\Delta(w+z(\tau))) \D \tau \D s \, \D t\\&\qquad
		+  \int_0^1 \int_0^1  s t^2 u_1\,u_2\int_0^1 \Delta (u_1-u_2) \p_q \p^2_u Q(x,w+z(\tau),\nabla(w+z(\tau)),\Delta(w+z(\tau))) \D \tau \D s \, \D t
		\\&\qquad 
		+ \int_0^1  \nabla (u_1-u_2) \cdot \int_0^1  \frac{d}{\D s}[\nabla_p Q(x,w+ stu_1,\nabla(w+stu_1),\Delta(w+stu_1))\,\D s\, \D t\\
		&\qquad +  \int_0^1  \nabla (u_1-u_2) \cdot \int_0^1 \frac{d}{\D s} [\nabla_p Q(x,w+stu_2,\nabla(w+stu_2),\Delta(w+stu_2)) \, \D s \,\D t\\& \qquad +  \int_0^1 \int_0^1  s t^2 \nabla u_1\otimes \nabla u_2 :\int_0^1 (u_1-u_2)\, \p_u\nabla_p^2 Q(x,w+z(\tau),\nabla(w+z(\tau)),\Delta(w+z(\tau))) \D \tau \D s \, \D t\\&\qquad+ 
		\int_0^1 st^2 \nabla u_1\otimes \nabla u_2 \int_0^1  \int_0^1 \nabla (u_1-u_2) \cdot \nabla^3_p Q(x,w+z(\tau),\nabla(w+z(\tau)),\Delta(w+z(\tau))) \D \tau \D s \, \D t\\&\qquad
		+  \int_0^1 \int_0^1  s t^2 \nabla u_1\otimes \nabla u_2:\int_0^1 \Delta (u_1-u_2) \p_q \nabla_p^2 Q(x,w+z(\tau),\nabla(w+z(\tau)),\Delta(w+z(\tau))) \D \tau \D s \, \D t
		\\&\qquad
		+ \int_0^1 \Delta(u_1-u_2) \int_0^1 \frac{d}{\D s}\p_q Q(x,w+stu_1,\nabla(w+stu_1),\Delta(w+stu_1))\D s \D t \\& \qquad +\int_0^1 \Delta (u_1-u_2) \int_0^1 \frac{d}{\D s}\p_q Q(x,w+stu_2,\nabla(w+stu_2),\Delta(w+stu_2))\D s \D t\\
		&\qquad+  \int_0^1 \int_0^1  s t^2 \Delta u_1\,\Delta u_2\int_0^1 (u_1-u_2)\, \p_u \p^2_q Q(x,w+z(\tau),\nabla(w+z(\tau)),\Delta(w+z(\tau))) \D \tau \D s \, \D t\\&\qquad+ 
		\int_0^1 \int_0^1  s t^2\Delta u_1\,\Delta u_2\int_0^1 \nabla (u_1-u_2) \cdot \nabla_p \p^2_q Q(x,w+z(\tau),\nabla(w+z(\tau)),\Delta(w+z(\tau))) \D \tau \D s\, \D t \\&\qquad + \int_0^1 \int_0^1  s t^2 \Delta u_1\,\Delta u_2\int_0^1 \Delta (u_1-u_2) \p^3_q Q(x,w+z(\tau),\nabla(w+z(\tau)),\Delta(w+z(\tau))) \D \tau \D s \, \D t.
	\end{align*}
	
\end{lemma}
\begin{proof}
 	We start with denoting $u_i=v+r_i$ for $i=1,2$, and write $R_v(r_1)-R_v(r_2)= I+II+III,$
	where 
	\begin{align*}
		I&=  \int_0^1 [\p_u Q(x,w+tu_1,\nabla(w+tu_1),\Delta(w+tu_1)) - \p_u Q(x,w,\nabla w,\Delta w)]u_1\\& \qquad -\int_0^1 [\p_u Q(x,w+tu_2,\nabla(w+tu_2),\Delta(w+tu_2)) - \p_u Q(x,w,\nabla w,\Delta w)]u_2\\
		II&=  \int_0^1 [\nabla_p Q(x,w+tu_1,\nabla(w+tu_1),\Delta(w+tu_1)) - \nabla_p Q(x,w,\nabla w,\Delta w)]\cdot \nabla u_1 \,\D t\\
		&\qquad -  \int_0^1 [\nabla_p Q(x,w+tu_2,\nabla(w+tu_2),\Delta(w+tu_2)) - \nabla_p Q(x,w,\nabla w,\Delta w)]\cdot \nabla u_2 \,\D t\\
		III&= \int_0^1 [\p_q Q(x,w+tu_1,\nabla(w+tu_1),\Delta(w+tu_1)) - \p_q Q(x,w,\nabla w,\Delta w)] \Delta u_1 \,\D t\\
		& \qquad - \int_0^1[\p_q Q(x,w+tu_2,\nabla(w+tu_2),\Delta(w+tu_2)) - \p_q Q(x,w,\nabla w,\Delta w)] \Delta u_2 \D t.
	\end{align*}
	We next express $I$ as
	\begin{align*}
			I&=  \int_0^1 [\p_u Q(x,w+tu_1,\nabla(w+tu_1),\Delta(w+tu_1)) - \p_u Q(x,w,\nabla w,\Delta w)]u_1\\& \qquad -\int_0^1 [\p_u Q(x,w+tu_2,\nabla(w+tu_2),\Delta(w+tu_2)) - \p_u Q(x,w,\nabla w,\Delta w)]u_2\\
			&=  \int_0^1 [\p_u Q(x,w+tu_1,\nabla(w+tu_1),\Delta(w+tu_1)) - \p_u Q(x,w,\nabla w,\Delta w)](u_1-u_2)\\& \qquad +\int_0^1 [\p_u Q(x,w+tu_2,\nabla(w+tu_2),\Delta(w+tu_2)) - \p_u Q(x,w,\nabla w,\Delta w)](u_1-u_2)\\
			&\qquad+   \int_0^1 [\p_u Q(x,w+tu_1,\nabla(w+tu_1),\Delta(w+tu_1)) - \p_u Q(x,w,\nabla w,\Delta w)]u_2 \\&\qquad- \int_0^1[\p_u Q(x,w+tu_2,\nabla(w+tu_2),\Delta(w+tu_2)) - \p_u Q(x,w,\nabla w,\Delta w)]u_1   \\
			&= \int_0^1 (u_1-u_2) \int_0^1 \frac{\D}{\D s}\p_u Q(x,w+stu_1,\nabla(w+stu_1),\Delta(w+stu_1)\dst \\& \qquad +\int_0^1 (u_1-u_2) \int_0^1 \frac{\D}{\D s}\p_u Q(x,w+stu_2,\nabla(w+stu_2),\Delta(w+stu_2)\, \dst \\
			&\qquad+   \int_0^1 u_2 \int_0^1 \frac{\D}{\D s}\p_u Q(x,w+stu_1,\nabla(w+stu_1),\Delta(w+stu_1))\dst   \\&\qquad- \int_0^1 u_1\int_0^1 \frac{\D}{\D s} \p_u Q(x,w+stu_2,\nabla(w+stu_2),\Delta(w+stu_2))\dst \\
			&= \int_0^1 (u_1-u_2) \int_0^1 \frac{\D}{\D s}\, \p_u Q(x,w+stu_1,\nabla(w+stu_1),\Delta(w+stu_1))\, \dst  \\& \qquad +\int_0^1 (u_1-u_2) \int_0^1 \frac{\D}{\D s}\p_u Q(x,w+stu_2,\nabla(w+stu_2),\Delta(w+stu_2))\, \dst \\
			&\qquad+   \int_0^1 u_2 \int_0^1\bigg( tu_1\,\p^2_u Q(x,w+stu_1,\nabla(w+stu_1),\Delta(w+stu_1))\D s \\&\qquad+ \nabla_p\p_u Q(x,w+stu_1,\nabla(w+stu_1),\Delta(w+stu_1)) \cdot t \nabla u_1 \\&\qquad+ \p_q\p_u Q(x,w+stu_1,\nabla(w+stu_1),\Delta(w+stu_1))\, t \Delta u_1 \bigg)\dst  \\&\qquad- \Bigg[ \int_0^1 u_1\int_0^1 \bigg( tu_2\, \p^2_u Q(x,w+stu_2,\nabla(w+stu_2),\Delta(w+stu_2))\\&\qquad+ \nabla_p\p_u Q(x,w+stu_2,\nabla(w+stu_2),\Delta(w+stu_2))\cdot t \nabla u_2 \\&\qquad+ \p_q\p_u Q(x,w+stu_2,\nabla(w+stu_2),\Delta(w+stu_2))\, t \Delta u_2\bigg)\, \dst \Bigg].
	\end{align*}
	This further entails
	\begin{align*}
		I &= \int_0^1 (u_1-u_2) \int_0^1 \frac{d}{ds}\p_u Q(x,w+stu_1,\nabla(w+stu_1),\Delta(w+stu_1))\dst \\& \qquad +\int_0^1 (u_1-u_2) \int_0^1 \frac{d}{ds}\p_u Q(x,w+stu_2,\nabla(w+stu_2),\Delta(w+stu_2))\dst\\
		&\qquad+   \int_0^1 \int_0^1 tu_1\,u_2\,[\p^2_u Q(x,w+stu_1,\nabla(w+stu_1),\Delta(w+stu_1))\\&\qquad-\p_u^2Q(x,w+stu_2,\nabla(w+stu_2),\Delta(w+stu_2))]\dst \\&\qquad+ \int_0^1 u_2 \int_0^1 \bigg(\nabla_p\p_u Q(x,w+stu_1,\nabla(w+stu_1),\Delta(w+stu_1)) \cdot t \nabla u_1 \\&\qquad+ \p_q\p_u Q(x,w+stu_1,\nabla(w+stu_1),\Delta(w+stu_1))\, t \Delta u_1 \bigg) \dst  \\&\qquad- \Bigg[ \int_0^1 u_1\int_0^1 \bigg( \nabla_p\p_u Q(x,w+stu_2,\nabla(w+stu_2),\Delta(w+stu_2))\cdot t \nabla u_2 \\&\qquad+ \p_q\p_u Q(x,w+stu_2,\nabla(w+stu_2),\Delta(w+stu_2))\, t \Delta u_2\bigg)\dst \Bigg].
	\end{align*}
	Denote $z(\tau):= stu_2+\tau\, s\,t( u_1-u_2)$.  Further simplifying $I$ using fundamental theory of calculus, we obtain 
	\begin{equation}\label{eq_I}
		\begin{aligned}
			I &= \int_0^1 (u_1-u_2) \int_0^1 \frac{d}{ds}\p_u Q(x,w+stu_1,\nabla(w+stu_1),\Delta(w+stu_1))\dst \\& \qquad +\int_0^1 (u_1-u_2) \int_0^1 \frac{d}{ds}\p_u Q(x,w+stu_2,\nabla(w+stu_2),\Delta(w+stu_2))\dst\\
			&\qquad+  \int_0^1 \int_0^1  s t^2 u_1\,u_2\int_0^1 (u_1-u_2)\, \p^3_u Q(x,w+z(\tau),\nabla(w+z(\tau)),\Delta(w+z(\tau))) \dstu\\&\qquad+ 
			\int_0^1 \int_0^1  s t^2 u_1\,u_2\int_0^1 \nabla (u_1-u_2) \cdot \nabla_p \p^2_u Q(x,w+z(\tau),\nabla(w+z(\tau)),\Delta(w+z(\tau))) \dstu\\&\qquad
			+  \int_0^1 \int_0^1  s t^2 u_1\,u_2\int_0^1 \Delta (u_1-u_2) \p_q \p^2_u Q(x,w+z(\tau),\nabla(w+z(\tau)),\Delta(w+z(\tau))) \dstu
			\\&\qquad+ \int_0^1 \int_0^1 t\,(u_2\nabla u_1- u_1\nabla u_2)\cdot\nabla_p\p_u Q(x,w+stu_1,\nabla(w+stu_1),\Delta(w+stu_1)) \dst \\&\qquad+ \int_0^1 \int_0^1  t(u_2 \Delta u_1- u_1\Delta u_2) \p_q\p_u Q(x,w+stu_1,\nabla(w+stu_1),\Delta(w+stu_1))\,\dst.
		\end{aligned}
	\end{equation}
	We now focus on $II$.
	\begin{align*}
		II&= \int_0^1 [\nabla_p Q(x,w+tu_1,\nabla(w+tu_1),\Delta(w+tu_1)) - \nabla_p Q(x,w,\nabla w,\Delta w)]\cdot \nabla u_1 \,\D t\\
		&\qquad -  \int_0^1 [\nabla_p Q(x,w+tu_2,\nabla(w+tu_2),\Delta(w+tu_2)) - \nabla_p Q(x,w,\nabla w,\Delta w)]\cdot \nabla u_2 \,\D t\\
		&= \int_0^1 [\nabla_p Q(x,w+tu_1,\nabla(w+tu_1),\Delta(w+tu_1)) - \nabla_p Q(x,w,\nabla w,\Delta w)]\cdot \nabla (u_1-u_2) \,\D t\\
		&\qquad +  \int_0^1 [\nabla_p Q(x,w+tu_2,\nabla(w+tu_2),\Delta(w+tu_2)) - \nabla_p Q(x,w,\nabla w,\Delta w)]\cdot \nabla (u_1- u_2) \,\D t\\& \qquad 
		+ \int_0^1 [\nabla_p Q(x,w+tu_1,\nabla(w+tu_1),\Delta(w+tu_1)) - \nabla_p Q(x,w,\nabla w,\Delta w)]\cdot \nabla u_2 \,\D t \\& \qquad 
		-  \int_0^1 [\nabla_p Q(x,w+tu_2,\nabla(w+tu_2),\Delta(w+tu_2)) - \nabla_p Q(x,w,\nabla w,\Delta w)]\cdot \nabla u_1 \,\D t\\
		&= \int_0^1  \nabla (u_1-u_2) \cdot \int_0^1  \frac{d}{ds}[\nabla_p Q(x,w+ stu_1,\nabla(w+stu_1),\Delta(w+stu_1))\,\dst \\
		&\qquad +  \int_0^1  \nabla (u_1-u_2) \cdot \int_0^1 \frac{d}{ds} [\nabla_p Q(x,w+stu_2,\nabla(w+stu_2),\Delta(w+stu_2)) \, \dst \\& \qquad 
		+ \int_0^1  \int_0^1 \bigg( t \nabla u_2 \cdot u_1\,\p_u \nabla_p Q(x,w+tsu_1,\nabla(w+tsu_1),\Delta(w+tsu_1)) \\& \qquad+ t \nabla u_1 \otimes\nabla u_2 : \cdot\nabla^2_p Q(x,w+tsu_1,\nabla(w+tsu_1),\Delta(w+tsu_1)) \\&\qquad+ t \Delta u_1\,\nabla u_2 \cdot \p_q\nabla_p Q(x,w+tsu_1,\nabla(w+tsu_1),\Delta(w+tsu_1))  \bigg) \dst \\& \qquad
		-  \Bigg[ \int_0^1  \int_0^1 \bigg( t \nabla u_1 \cdot u_2\,\p_u \nabla_p Q(x,w+tsu_2,\nabla(w+tsu_2),\Delta(w+tsu_2)) \\& \qquad+ t \nabla u_2 \otimes\nabla u_1 : \cdot\nabla^2_p Q(x,w+tsu_2,\nabla(w+tsu_2),\Delta(w+tsu_2)) \\&\qquad+ t \Delta u_2\,\nabla u_1 \cdot \p_q\nabla_p Q(x,w+tsu_2,\nabla(w+tsu_2),\Delta(w+tsu_2))  \bigg) \dst\Bigg].
	\end{align*}
	Similarly, we can write $III$
	\begin{align*}
		III& = \int_0^1 [\p_q Q(x,w+tu_1,\nabla(w+tu_1),\Delta(w+tu_1)) - \p_q Q(x,w,\nabla w,\Delta w)] \Delta u_1 \,\D t\\
		& \qquad - \int_0^1[\p_q Q(x,w+tu_2,\nabla(w+tu_2),\Delta(w+tu_2)) - \p_q Q(x,w,\nabla w,\Delta w)] \Delta u_2 \D t\\
		&= \int_0^1 \Delta(u_1-u_2) \int_0^1 \frac{d}{ds}\p_q Q(x,w+stu_1,\nabla(w+stu_1),\Delta(w+stu_1))\dst \\& \qquad +\int_0^1 \Delta(u_1-u_2) \int_0^1 \frac{d}{ds}\p_q Q(x,w+stu_2,\nabla(w+stu_2),\Delta(w+stu_2))\dst\\
		&\qquad+  \int_0^1 \Delta u_2 \int_0^1\bigg( tu_1\,\p_u\p_q Q(x,w+stu_1,\nabla(w+stu_1),\Delta(w+stu_1))\D s \\&\qquad+ \nabla_p\p_q Q(x,w+stu_1,\nabla(w+stu_1),\Delta(w+stu_1)) \cdot t \nabla u_1 \\&\qquad+ \p^2_q Q(x,w+stu_1,\nabla(w+stu_1),\Delta(w+stu_1))\, t \Delta u_1 \bigg) \dst  \\&\qquad- \Bigg[ \int_0^1 \Delta u_1\int_0^1 \bigg( tu_2\, \p_u\p_q Q(x,w+stu_2,\nabla(w+stu_2),\Delta(w+stu_2))\\&\qquad+ \nabla_p\p_q Q(x,w+stu_2,\nabla(w+stu_2),\Delta(w+stu_2))\cdot t \nabla u_2 \\&\qquad+ \p^2_q Q(x,w+stu_2,\nabla(w+stu_2),\Delta(w+stu_2))\, t \Delta u_2\bigg)\dst\Bigg].
	\end{align*}
	Further simplifying $II$ and $III$ we conclude 
	\begin{equation}\label{eq_II}
		\begin{aligned}
			II&=  \int_0^1  \nabla (u_1-u_2) \cdot \int_0^1  \frac{d}{ds}[\nabla_p Q(x,w+ stu_1,\nabla(w+stu_1),\Delta(w+stu_1))\,\dst\\
			&\qquad +  \int_0^1  \nabla (u_1-u_2) \cdot \int_0^1 \frac{d}{ds} [\nabla_p Q(x,w+stu_2,\nabla(w+stu_2),\Delta(w+stu_2)) \, \dst\\& \qquad +  \int_0^1 \int_0^1  s t^2 \nabla u_1\otimes \nabla u_2 :\int_0^1 (u_1-u_2)\, \p_u\nabla_p^2 Q(x,w+z(\tau),\nabla(w+z(\tau)),\Delta(w+z(\tau))) \dstu\\&\qquad+ 
			\int_0^1 st^2 \nabla u_1\otimes \nabla u_2 \int_0^1  \int_0^1 \nabla (u_1-u_2) \cdot \nabla^3_p Q(x,w+z(\tau),\nabla(w+z(\tau)),\Delta(w+z(\tau))) \dstu\\&\qquad
			+  \int_0^1 \int_0^1  s t^2 \nabla u_1\otimes \nabla u_2:\int_0^1 \Delta (u_1-u_2) \p_q \nabla_p^2 Q(x,w+z(\tau),\nabla(w+z(\tau)),\Delta(w+z(\tau))) \dstu
			\\&\qquad+ \int_0^1 \int_0^1 t\,(u_1\nabla u_2- u_2\nabla u_1)\cdot\nabla_p\p_u Q(x,w+stu_1,\nabla(w+stu_1),\Delta(w+stu_1))\, \dst\\&\qquad+ \int_0^1 \int_0^1  t( \Delta u_1\, \nabla u_2- \Delta u_2\, \nabla u_1) \cdot\p_q\nabla_p Q(x,w+stu_1,\nabla(w+stu_1),\Delta(w+stu_1))\, \dst,
		\end{aligned}
	\end{equation}
	and 
	\begin{equation}\label{eq_III}
		\begin{aligned}
			III&= \int_0^1 \Delta(u_1-u_2) \int_0^1 \frac{d}{ds}\p_q Q(x,w+stu_1,\nabla(w+stu_1),\Delta(w+stu_1))\dst \\& \qquad +\int_0^1 \Delta (u_1-u_2) \int_0^1 \frac{d}{ds}\p_q Q(x,w+stu_2,\nabla(w+stu_2),\Delta(w+stu_2))\dst\\
			&\qquad+  \int_0^1 \int_0^1  s t^2 \Delta u_1\,\Delta u_2\int_0^1 (u_1-u_2)\, \p_u \p^2_q Q(x,w+z(\tau),\nabla(w+z(\tau)),\Delta(w+z(\tau))) \dstu\\&\qquad+ 
			\int_0^1 \int_0^1  s t^2\Delta u_1\,\Delta u_2\int_0^1 \nabla (u_1-u_2) \cdot \nabla_p \p^2_q Q(x,w+z(\tau),\nabla(w+z(\tau)),\Delta(w+z(\tau))) \dstu \\&\qquad + \int_0^1 \int_0^1  s t^2 \Delta u_1\,\Delta u_2\int_0^1 \Delta (u_1-u_2) \p^3_q Q(x,w+z(\tau),\nabla(w+z(\tau)),\Delta(w+z(\tau))) \dstu
			\\&\qquad+\int_0^1 \int_0^1  t( \Delta u_2\, \nabla u_1- \Delta u_1\, \nabla u_2) \cdot\p_q\nabla_p Q(x,w+stu_1,\nabla(w+stu_1),\Delta(w+stu_1))\, \dst\\&\qquad+ \int_0^1 \int_0^1  t(u_1\Delta u_2- u_2\Delta u_1) \p_q\p_u Q(x,w+stu_1,\nabla(w+stu_1),\Delta(w+stu_1))\, \dst.
		\end{aligned}
	\end{equation}
	Next combining \eqref{eq_I}, \eqref{eq_II}, and \eqref{eq_III} we obtain
	\begin{align*}
			& R_v(r_1) -R_v(r_2)=     I+II+III\\
			&=\int_0^1 (u_1-u_2) \int_0^1 \frac{d}{ds}\p_u Q(x,w+stu_1,\nabla(w+stu_1),\Delta(w+stu_1))\dst \\& \qquad +\int_0^1 (u_1-u_2) \int_0^1 \frac{d}{ds}\p_u Q(x,w+stu_2,\nabla(w+stu_2),\Delta(w+stu_2))\dst\\
			&\qquad+  \int_0^1 \int_0^1  s t^2 u_1\,u_2\int_0^1 (u_1-u_2)\, \p^3_u Q(x,w+z(\tau),\nabla(w+z(\tau)),\Delta(w+z(\tau))) \dstu\\&\qquad+ 
			\int_0^1 \int_0^1  s t^2 u_1\,u_2\int_0^1 \nabla (u_1-u_2) \cdot \nabla_p \p^2_u Q(x,w+z(\tau),\nabla(w+z(\tau)),\Delta(w+z(\tau))) \dstu\\&\qquad
			+  \int_0^1 \int_0^1  s t^2 u_1\,u_2\int_0^1 \Delta (u_1-u_2) \p_q \p^2_u Q(x,w+z(\tau),\nabla(w+z(\tau)),\Delta(w+z(\tau))) \dstu
			\\&\qquad+ \underbrace{\int_0^1 \int_0^1 t\,(u_2\nabla u_1- u_1\nabla u_2)\cdot\nabla_p\p_u Q(x,w+stu_1,\nabla(w+stu_1),\Delta(w+stu_1)) \dst}_{A}\\&\qquad+ \underbrace{ \int_0^1 \int_0^1  t(u_2 \Delta u_1- u_1\Delta u_2) \p_q\p_u Q(x,w+stu_1,\nabla(w+stu_1),\Delta(w+stu_1))\, \dst }_{B}\\&\qquad
			+ \int_0^1  \nabla (u_1-u_2) \cdot \int_0^1  \frac{\D }{\D s}[\nabla_p Q(x,w+ stu_1,\nabla(w+stu_1),\Delta(w+stu_1))\,\dst\\
			&\qquad +  \int_0^1  \nabla (u_1-u_2) \cdot \int_0^1 \frac{\D}{\D s} [\nabla_p Q(x,w+stu_2,\nabla(w+stu_2),\Delta(w+stu_2)) \, \dst\\& \qquad +  \int_0^1 \int_0^1  s t^2 \nabla u_1\otimes \nabla u_2 :\int_0^1 (u_1-u_2)\, \p_u\nabla_p^2 Q(x,w+z(\tau),\nabla(w+z(\tau)),\Delta(w+z(\tau))) \dstu\\&\qquad+ 
			\int_0^1 st^2 \nabla u_1\otimes \nabla u_2 \int_0^1  \int_0^1 \nabla (u_1-u_2) \cdot \nabla^3_p Q(x,w+z(\tau),\nabla(w+z(\tau)),\Delta(w+z(\tau))) \dstu\\&\qquad
			+  \int_0^1 \int_0^1  s t^2 \nabla u_1\otimes \nabla u_2:\int_0^1 \Delta (u_1-u_2) \p_q \nabla_p^2 Q(x,w+z(\tau),\nabla(w+z(\tau)),\Delta(w+z(\tau))) \dstu
			\\&\qquad+\underbrace{ \int_0^1 \int_0^1 t\,(u_1\nabla u_2- u_2\nabla u_1)\cdot\nabla_p\p_u Q(x,w+stu_1,\nabla(w+stu_1),\Delta(w+stu_1)) \dst}_{-A}\\&\qquad+ \underbrace{\int_0^1 \int_0^1  t( \Delta u_1\, \nabla u_2- \Delta u_2\, \nabla u_1) \cdot\p_q\nabla_p Q(x,w+stu_1,\nabla(w+stu_1),\Delta(w+stu_1))\, \dst}_{C}\\&\qquad + \int_0^1 \Delta(u_1-u_2) \int_0^1 \frac{\D}{\D s}\p_q Q(x,w+stu_1,\nabla(w+stu_1),\Delta(w+stu_1))\dst \\& \qquad +\int_0^1 \Delta (u_1-u_2) \int_0^1 \frac{\D }{\D s}\p_q Q(x,w+stu_2,\nabla(w+stu_2),\Delta(w+stu_2))\dst\\
			&\qquad+  \int_0^1 \int_0^1  s t^2 \Delta u_1\,\Delta u_2\int_0^1 (u_1-u_2)\, \p_u \p^2_q Q(x,w+z(\tau),\nabla(w+z(\tau)),\Delta(w+z(\tau))) \dstu\\&\qquad+ 
			\int_0^1 \int_0^1  s t^2\Delta u_1\,\Delta u_2\int_0^1 \nabla (u_1-u_2) \cdot \nabla_p \p^2_q Q(x,w+z(\tau),\nabla(w+z(\tau)),\Delta(w+z(\tau))) \dstu \\&\qquad + \int_0^1 \int_0^1  s t^2 \Delta u_1\,\Delta u_2\int_0^1 \Delta (u_1-u_2) \p^3_q Q(x,w+z(\tau),\nabla(w+z(\tau)),\Delta(w+z(\tau))) \dstu
			\\&\qquad+ \underbrace{\int_0^1 \int_0^1  t( \Delta u_2\, \nabla u_1- \Delta u_1\, \nabla u_2) \cdot\p_q\nabla_p Q(x,w+stu_1,\nabla(w+stu_1),\Delta(w+stu_1))\, \dst}_{-C}\\&\qquad+ \underbrace{\int_0^1 \int_0^1  t(u_1\Delta u_2- u_2\Delta u_1) \p_q\p_u Q(x,w+stu_1,\nabla(w+stu_1),\Delta(w+stu_1))\, \dst}_{-B}.
	\end{align*}
	Simplifying above we conclude the proof.

\end{proof}

\bibliography{bibliography}
\bibliographystyle{plain}

\end{document}